\documentclass[11pt,reqno]{amsart}
\usepackage[shortlabels]{enumitem}
\usepackage{esint}
\usepackage{physics}
\usepackage{relsize}
\usepackage[backref]{hyperref}
\usepackage{mathtools}
\usepackage{amsfonts}
\usepackage[all]{xy}
\usepackage{geometry}
\usepackage{amsmath,amssymb} 
\usepackage{faktor}
\usepackage{amscd}
\usepackage{tikz-cd}

\newtheorem{theorem}{Theorem}[section]
\newtheorem{lemma}[theorem]{Lemma}
\newtheorem{definition}[theorem]{Definition}

\theoremstyle{corollary}
\newtheorem{corollary}[theorem]{Corollary}
\theoremstyle{conjecture}
\newtheorem{conjecture}[theorem]{Conjecture}

\theoremstyle{assumption}

\theoremstyle{proposition}
\newtheorem{proposition}[theorem]{Proposition}
\theoremstyle{remark}
\newtheorem{remark}[theorem]{Remark}
\numberwithin{equation}{section}
\everymath{\displaystyle}

\newcommand{\ii}{\ensuremath{\sqrt{-1}}}
\newcommand{\pp}{\bar\partial}
\newcommand{\p}{\partial}

\DeclareMathOperator{\ric}{Ric}
\DeclareMathOperator{\nef}{nef}

\DeclareMathOperator{\BC}{BC}

\DeclareMathOperator{\supp}{supp}

\DeclareMathOperator{\Exc}{Exc}
\DeclareMathOperator{\nd}{nd}
\makeatletter
\newcommand*{\da@rightarrow}{\mathchar"0\hexnumber@\symAMSa 4B }
\newcommand*{\da@leftarrow}{\mathchar"0\hexnumber@\symAMSa 4C }
\newcommand*{\xdashrightarrow}[2][]{%
  \mathrel{%
    \mathpalette{\da@xarrow{#1}{#2}{}\da@rightarrow{\,}{}}{}%
  }%
}
\newcommand{\xdashleftarrow}[2][]{%
  \mathrel{%
    \mathpalette{\da@xarrow{#1}{#2}\da@leftarrow{}{}{\,}}{}%
  }%
}
\newcommand*{\da@xarrow}[7]{%
  \sbox0{$\ifx#7\scriptstyle\scriptscriptstyle\else\scriptstyle\fi#5#1#6\m@th$}%
  \sbox2{$\ifx#7\scriptstyle\scriptscriptstyle\else\scriptstyle\fi#5#2#6\m@th$}%
  \sbox4{$#7\dabar@\m@th$}%
  \dimen@=\wd0 %
  \ifdim\wd2 >\dimen@
    \dimen@=\wd2 %
  \fi
  \count@=2 %
  \def\da@bars{\dabar@\dabar@}%
  \@whiledim\count@\wd4<\dimen@\do{%
    \advance\count@\@ne
    \expandafter\def\expandafter\da@bars\expandafter{%
      \da@bars
      \dabar@ 
    }%
  }%
  \mathrel{#3}%
  \mathrel{%
    \mathop{\da@bars}\limits
    \ifx\\#1\\%
    \else
      _{\copy0}%
    \fi
    \ifx\\#2\\%
    \else
      ^{\copy2}%
    \fi
  }%
  \mathrel{#4}%
}
\makeatother

\begin{document}

\title[Weak base-point freeness and diameter lower bound]{Weak transcendental base-point freeness and diameter lower bounds for the K\"ahler-Ricci flow}

\author{Junsheng Zhang}
\address{Courant Institute of Mathematical Sciences\\
  New York University, 251 Mercer St\\
  New York, NY 10012\\}
\curraddr{}
\email{jz7561@nyu.edu}

\begin{abstract}
We prove a weaker version of the transcendental base-point freeness  on compact K\"ahler manifolds. As a consequence, we derive the diameter lower bound for finite time singularities of K\"ahler-Ricci flow with non-Fano initial data.
 \end{abstract}

\maketitle

\section{Introduction}
Let $(X,\omega_0)$ be a compact K\"ahler manifold. 
We consider the K\"ahler–Ricci flow on $X$, which develops a finite-time singularity and, after rescaling, we may assume occurs at time $1$, that is, on $X \times [0,1)$ we have
\begin{equation}\label{eq--ric flow}
\begin{cases}
\partial_t \omega = -\operatorname{Ric}(\omega), \\[6pt]
\omega|_{t=0} = \omega_0.
\end{cases}
\end{equation}
By the work in \cite{tian-zhang2006,tsuji1988}, we know that  \eqref{eq--ric flow} develops a  singularity at time 1 if and only if 
\begin{equation}
\mathrm{nef}_X(\omega_0):=\sup\{t>0\mid \omega_0+tK_X \text{ is a K\"ahler class}\}=1.
\end{equation}

Throughout this paper, we say that $(X, \omega_0)$ is Fano, if $X$ is a Fano manifold and $\omega_0\in c_1(X)$. 
Our main result establishes the following diameter lower bound for the K\"ahler–Ricci flow starting from non-Fano initial data. 
\begin{theorem}\label{thm--main}
    Suppose $(X, \omega_0)$ is not Fano and $\nef_X(\omega_0)=1$. Let $\omega_t$ be the solution of the K\"ahler-Ricci-flow \eqref{eq--ric flow}.
    Then
    \begin{equation}
        \liminf_{t\nearrow 1}\operatorname{diam}(X, \omega_t) >0.
    \end{equation}
   
\end{theorem}
Note that if \( (X, \omega_0) \) is Fano, then by Perelman’s work (see \cite{sesumtian}), we have
\[
\operatorname{diam}(X, \omega_t) \le C \sqrt{1 - t}.
\]
Therefore, Theorem~\ref{thm--main} is equivalent to the statement that for a Kähler–Ricci flow developing a finite-time singularity at \( t = 1 \), \( (X, \omega_0) \) is Fano if and only if the flow becomes extinct at time 1, that is,
\[
\operatorname{diam}(X, \omega_t) \to 0 \quad \text{as } t \nearrow 1.
\]
This resolves the conjectures proposed in \cite[Conjecture~1.2]{song2014finite} and \cite[Conjecture~1.3]{tosatti2-zhang}, which are related to \cite[Conjecture~4.4]{tian2008new}.

It is known from the works \cite{tosatti2-zhang} (see also \cite{songtian2007, song2014finite}) that, by a Schwarz lemma–type argument, the diameter lower bound would follow from the folklore \emph{transcendental base-point-freeness} conjecture; see, for instance, \cite[Conjecture~1.2]{tosatti2-zhang}, \cite[Conjecture~1.2]{filip-tosatti}, and Conjecture~\ref{conj--basepointfree} in this paper.
The diameter lower bound has been established in the following cases:
\begin{itemize}
    \item when $X$ is projective and  $[\omega_0] \in H^2(X,\mathbb{Q})$, by \cite{song2014finite};
    \item when $\dim X \leq 3$, by \cite{tosatti2-zhang};
    \item when $X$ is projective, by \cite{das-hacon2024}.
\end{itemize}
Our approach follows the same general outline, but we must overcome the difficulty that the general transcendental base-point-freeness conjecture remains out of reach.  
Indeed, we establish the following weaker bimeromorphic version of the transcendental base-point-freeness.
\begin{theorem}\label{thm-weak basepointfree}
    Let $X$ be a compact K\"ahler manifold. Suppose $\alpha$ is a big and nef class and $K_X+\alpha$ is pseudo-effective. Then there exists projective fibration $X\xlongleftarrow{\mu} Z\xlongrightarrow{f}Y$ between compact K\"ahler manifolds such that $\mu$ is a composition of blow-ups along smooth holomorphic submanifolds,  general fibers of $f$ are rationally connected and
         \begin{equation}
             \mu^*(K_X+\alpha)=f^*(\alpha_1)+D,
         \end{equation} where $ \alpha_1$ is a K\"ahler class and $D$ is an effective $\mathbb R$-divisor.
      Moreover if $K_X+\alpha$ is nef, then there exists an effective $\mathbb Q$-divisor $D_Y$ on $Y$ and an effective $\mu$-exceptional $\mathbb R$-divisor $E$ such that 
      \begin{equation}
          E+D=f^*(D_Y).
      \end{equation}
\end{theorem}
We show in Proposition~\ref{prop-dim of base coincide with numerical} that 
\( \dim Y \) coincides with the numerical dimension of \( K_X + \alpha \) 
as defined in \cite{bdpp}.  
The assumptions on \( X \) and \( \alpha \) can be weaken to the case where 
\( \alpha \) is big and \( (X, \alpha) \) forms a generalized klt pair; 
see Definition~\ref{def-numerical generalized pair} and 
Theorem~\ref{thm--generalized klt version}, as well as 
\cite{dhy, hacon-paun} for further discussions on generalized pairs.  
Moreover this result suggests that, in order to establish the transcendental 
base-point-freeness conjecture, it should suffice to treat the case where 
the adjoint class \( K_X + \alpha \) is big; see 
Remark~\ref{rmk-enough to show freeness for big class} and Corollary \ref{cor-numerical dim at most 3 is done}. In particular, applying \cite[Theorems~1.2 and~2.33]{dhy}, we can show that the transcendental base-point-freeness theorem holds whenever the adjoint class has numerical dimension at most three.

\begin{corollary}\label{cor-bpf for ndleq 3}
Let \( X \) be a compact Kähler manifold, and let \( \alpha \) be a big and nef class such that \( K_X + \alpha \) is nef. 
If the numerical dimension \( \nd(K_X + \alpha) \le 3 \), then there exists a holomorphic morphism 
\( f : X \to Y \) onto a compact Kähler space and a Kähler class \( \omega_Y \) on \( Y \) such that
\begin{equation}
    K_X + \alpha = f^*(\omega_Y).
\end{equation}
\end{corollary}
 
Note that the statement in Theorem \ref{thm-weak basepointfree} for the nef case of \( K_X + \alpha \) would follow as a consequence of the transcendental base-point-freeness conjecture.  
However, the pseudoeffective case of \( K_X + \alpha \) in this transcendental setting seems to be new.  
When \( X \) is projective and \( \alpha \) is defined by an \( \mathbb{R} \)-Cartier divisor, the theorem above corresponds to a weaker version of the non-vanishing theorem proved in \cite{BCHM} and \cite[Theorem~0.1]{puaun2012relative}.

The proof of Theorem~\ref{thm-weak basepointfree} relies on the results of \cite{ou2025, bdpp} and on the relative minimal model program (MMP) for projective morphisms, which has been developed in \cite{fujino2022, fujino-cone, Das_HP2024}, extending the work of \cite{BCHM} to projective morphisms between normal complex analytic spaces.  
We briefly outline the main idea of the proof here.  

If \( K_X + \alpha \) is a big class, we may simply take \( f \) to be the identity, and the existence of \( \mu \) satisfying the desired property follows directly from Demailly’s regularization theorem~\cite{demailly1992, demailly2012analytic} and the principalization of ideal sheaves \cite{Hironaka1, wlodarczyk2008}.  
Hence, we focus on the case~\eqref{eq--initial assumption}, under which, in particular, \( K_X \) is not pseudoeffective:
\begin{equation}\label{eq--initial assumption}
    K_X + \alpha \text{ is not big.}
\end{equation}
In this situation, we invoke the recent breakthrough result of \cite{ou2025}, which generalizes the theorem of \cite{bdpp} to the non-projective Kähler setting.  
It follows that for any compact Kähler manifold \( X \) with non-pseudoeffective canonical bundle \( K_X \), there exists a bimeromorphic map
\(
\mu_0 : X_0 \longrightarrow X,
\)
which is a composition of blow-ups along smooth complex submanifolds, such that \( X_0 \) admits a projective fibration to a lower-dimensional compact Kähler manifold
\(
X_0 \longrightarrow Y_0,
\)
whose general fibers are rationally connected and whose canonical class \( K_{Y_0} \) is pseudoeffective.  
This construction serves as the starting point for the proof of Theorem~\ref{thm--main}.

Ideally, one would hope that \( \mu_0^*(K_X + \alpha) \) is the pullback of a big class on \( Y_0 \), in which case the theorem would follow immediately by applying Demailly’s regularization theorem on \( Y_0 \).  
In general, however, we need to run relative (i.e., over \( Y_0 \)) MMPs with scaling; see \cite{nakayama1987, fujino2022, Das_HP2024}.  
It is worth noting that our argument requires running two types of relative MMPs with scaling.  
The first step is roughly to run a relative \( K_{X_0} + \mu_0^*(\alpha) \)-MMP with scaling of a relatively ample \( \mathbb{R} \)-line bundle, in order to make \( K_{X_0} + \mu_0^*(\alpha) \) nef.  
Once this is achieved, we then run the relative \( K_X \)-MMP with scaling of the big class \( \mu_0^*(\alpha) \).
The second, relative MMP differs from the standard MMP, which typically stops upon reaching a Mori–Fano fibration.  
Here, even after arriving at a Mori–Fano fibration, as long as the condition~\eqref{eq--initial assumption} continues to hold on the base of the fibration, we need to continue the relative MMP on the base of the Mori-Fano fibration.  
For this step, we rely on the results of \cite{hacon-paun}, which ensure that one still obtains a big boundary class on the base after reaching a Mori–Fano fibration.

Since \( K_{Y_0} \) is pseudoeffective, we can show that after finitely many 
divisorial contractions, flips, and Mori–Fano fibrations, the process 
terminates at a Mori–Fano fibration whose base no longer satisfies 
\eqref{eq--initial assumption}.  
At this stage, we apply Demailly’s regularization theorem to a resolution of 
the resulting base and trace back the adjoint class.  
Finally, the stronger statement in Theorem~\ref{thm-weak basepointfree} for 
the case where \( K_X + \alpha \) is nef follows directly from a standard 
application of the negativity lemma.

When \( \nef_X(\omega_0) \ge 1 \), Theorem~\ref{thm-weak basepointfree} provides a positive \((1,1)\)-current representing the cohomology class \( \mu^*(K_X + \omega_0) \) with suitable regularity properties.  
This current plays a central role in establishing a generalized parabolic Schwarz lemma, which in turn yields a lower bound for the diameter of \( \omega_t \).  
The proof of this generalized Schwarz lemma requires addressing some analytic issues, notably those arising from the exceptional locus of \( \mu \) and the absence of a uniform \( C^0 \)-control for the potential; see Section~\ref{sec--proof of main theorem}.

The paper is organized as follows. In Section~\ref{sec-relative mmp} we collect several preliminary results, prove a number of auxiliary lemmas used throughout the paper, and give a brief sketch of the relative MMP with scaling of a relatively big $\mathbb{R}$-line bundle. In Section~\ref{sec-structur theorem} we establish the weak transcendental base-point-freeness theorem and a generalized version thereof, and we discuss some of its consequences. Finally, in Section~\ref{sec--proof of main theorem} we present the proof of Theorem~\ref{thm--main}: we first derive a degenerate $C^0$ estimate for the potential and then complete the argument via a generalized parabolic Schwarz lemma. We also remark that an analogous diameter lower bound holds along the continuity paths of K\"ahler metrics considered in the literature.
\subsection*{Acknowledgements}The author would like to thank Valentino Tosatti for many helpful discussions and for his constant support. He thanks Christopher Hacon for generous email exchanges and for answering the author’s questions. He also thanks Mihai Păun and Song Sun for their interest.
\section{Relative MMP for  projective morphisms}\label{sec-relative mmp}
A \emph{fibration} between normal complex analytic spaces is a proper surjective morphism with connected fibers.  
A \emph{projective fibration} between compact normal complex analytic spaces is a fibration that is moreover a projective morphism. A \textit{pair} \( (Z, B) \) consists of a normal complex analytic space \( Z \) together with an effective \( \mathbb{Q} \)-divisor \( B \) such that \( K_Z + B \) is \( \mathbb{Q} \)-Cartier; that is, there exists an integer \( m \ge 1 \) such that  
\begin{equation}
    \bigl((\mathcal{O}_Z(K_Z) \otimes \mathcal{O}_Z(B))^{\otimes m}\bigr)^{**}
\end{equation}
is a line bundle, where \( \mathcal{O}_Z(K_Z) \) denotes the canonical sheaf of \( Z \).
We refer to \cite[Section 4.6]{boucksom-guedj-book}, \cite[Sections~2 and~3.1]{horing-peternell}, \cite[Section~2]{dh-3fold} and \cite[Section~2]{Das_HP2024} for the basic definitions of K\"ahler spaces, $\mathbb Q$-facotriality and strong $\mathbb Q$-facotriality  of a normal complex analytic space, the Bott--Chern cohomology group and its basic properties, as well as for the notions of K\"ahler classes, nef classes, big classes and pseudo-effective classes on normal compact K\"ahler spaces.

\subsection{Prelininaries}\label{sec:preliminary}
Throughout this section, we will mainly consider projective fibrations $f:Z\rightarrow Y$ between compact klt K\"ahler spaces. We collect some basic facts here for later references.
\begin{enumerate}
\item Let \( Z \) be a compact normal Kähler space, and let \( \nu: Z' \rightarrow Z \) be a projective morphism.  
Then \( Z' \) is  a Kähler space \cite[Lemma 4.4]{fujiki1978}.  
Moreover, every compact normal Kähler space \( Z \) admits a projective resolution \( \nu: Z' \rightarrow Z \) by \cite[Theorem 5.4.2]{ArocaVincen-Hironaka} and \cite[Corollary 2]{hironaka1975flattening}.  
Then in particular, \( Z' \) is a compact Kähler manifold.
    \item klt singularities imply raitonal singularities \cite[Theorem 3.12]{fujino2022};
    \item if general fibers $F$ of the morphism $f:Z\rightarrow Y$ are rationally connected, then by \cite[Theorem 4.1]{claudon-horing}, for all $i>0$, \begin{equation}\label{eq--vanishing of high pushforward}
        R^if_*\mathcal{O}_Z=0;
    \end{equation}
       \item suppose we have a commutative diagram of projective fibrations
\begin{equation}  
\begin{tikzcd}[column sep=small, row sep=small]
Z \arrow[dr,"f"'] \arrow[rr,"\varphi"] && Z_1 \arrow[dl,"f_1"] \\
& D &
\end{tikzcd}
\end{equation}If a general fiber of $f$ is rationally conected, then general fibers of \( f_1 \) and a general fibers of $\varphi_1$ are also rationally connected. Moreover,  a general fiber of \( f_1 \) is rationally connected also when \( \varphi \) is a bimeromorphic map. In the following, we will use this property without explicitly mentioning it.
    \item by \cite[Lemma 2.5]{das-hacon2024} and  the proof in \cite[Lemma 2.6]{das-hacon2024}, we know that   $f^*:H^{1,1}_{\BC}(Y)\rightarrow H^{1,1}_{\BC}(X)$ is injective. Moreover if \eqref{eq--vanishing of high pushforward} is satisfied, then the image of $f^*$ is given by 
\begin{equation*}\operatorname{Im}\left(f^*\right)=\left\{\alpha \in H_{\mathrm{\BC}}^{1,1}(X) \mid \alpha \cdot C=0\right. \text{ for all curves $C \subset X$ s.t. $\left.f(C)=\mathrm{pt}\right\}$};
    \end{equation*}

    \item Let $N_1(Z/Y)$ be the real vector space generated by the irreducible curves on $Z$ whose image under $f$ is a point, module the numerical equivalence relation that 
    \begin{equation}\label{eq-nuemrical equi}
        C_1\equiv C_2 \text{ if and only if } C_1\cdot L=C_2\cdot L_2 \text{ for any } L\in \mathrm{Pic}(Z).
    \end{equation}Let $N^1(Z/Y)$ be the real vector space generated by line bundles on $Z$ module the numerical equivalence that 
    \begin{equation}
        L_1\equiv L_2 \text{ if and only if } C\cdot L_1=C\cdot L_2 \text{ for any curve $C\in N_1(Z/Y)$ }.  
    \end{equation}By \cite[Proposition 4.3]{nakayama1987}, we know that both $N_1(Z/Y)$ and $N^1(Z/Y)$ are finite dimensional real vector spaces. Throughout this paper, by a relative $\mathbb{R}$-line bundle we mean an element of $N^1(Z/Y)$; that is, we always identify it with its numerical equivalence class in $N^1(Z/Y)$.

    \item By \cite[Theorem 12.1.3, Corollary 12.1.5]{kollar-mori1992classification} and (3) above, we know that if \eqref{eq--vanishing of high pushforward} is satisfied, then  
    \begin{equation}\label{eq--divisor equal H2}
        N_1(Z/Y)=H_2(Z/Y,\mathbb R) \text{ and } N^1(Z/Y)=H^{1,1}_{\BC}(Z)/H^{1,1}_{\BC}(Y).
    \end{equation}Indeed, under the assumption \eqref{eq--vanishing of high pushforward}, we can follows the argument in the proof of \cite[Lemma 2.6]{das-hacon2024} and we can get inclusions 
    \begin{equation}
      N^1(Z/Y)  \hookrightarrow H^{1,1}_{\BC}(Z)/  H^{1,1}_{\BC}(Y)\hookrightarrow H^2(Z,\mathbb R)/H^2(Y, \mathbb R).    \end{equation}Taking a projective resolution \( Z' \rightarrow Z \), and applying the conclusion to both \( Z' \rightarrow Z \) and \( Z' \rightarrow Y \), we are reduced to show \eqref{eq--divisor equal H2} holds in the case where \( Z \) is smooth.
      By \cite[Theorem 12.1.3]{kollar-mori1992classification}, we know that $\ker(H_2(Z,\mathbb R)\rightarrow H_2(Y,\mathbb R))$ is generated by 1-cycles contained in fibers of $f$. Therefore in order to show \eqref{eq--divisor equal H2}, it is enough to show that if a 1-cycle is numerically zero, then it also defines a zero homology class in $H_2(Z,\mathbb R)$. Since $Z$ is a smooth compact K\"ahler manifold, this follows from the  Lefschetz theorem on (1,1)-classes. 
    \item As usual, we can define the Mori cone
$\overline{\mathrm{NE}}(Z / Y)
$,  that is, $\overline{\mathrm{NE}}(Z / Y )$ is the closure of the convex cone in $N_1(Z / Y )$ spanned by the integral curves $C$ on $Z$ such that $\pi(C)$ is a point of $Y$. Then we have the following  Kleiman’s ampleness criterion for projective morphisms \cite[Proposition 4.7]{nakayama1987}:
   an   $\mathbb R$-line bundle  $L$ is ample over $Y$ if and only if $L$ is positive on $\overline{\mathrm{NE}}(Z / Y) \backslash\{0\}$.

\item A class \( \alpha \in H^{1,1}_{\BC}(Z) \) is said to be \emph{\( f \)-big} (or \emph{big over \( Y \)}) if there exists a class \( \theta \in H^{1,1}_{\BC}(Y) \) such that \( \alpha + f^*(\theta) \) is a big class on \( Z \).  
A class \( \alpha \in H^{1,1}_{\BC}(Z) \) is said to be \emph{\( f \)-pseudoeffective} if for every \( f \)-big class \( \beta \), the sum \( \alpha + \beta \) is \( f \)-big. 
We will need the following well-known result.
\begin{lemma}\label{lem-characterization of big divisors}\label{lem-relative big line bundle has fiberwise sections}
   Let $L$ be an $f$-big $\mathbb{Q}$-line bundle on $Z$. For a general point $y\in Y$, we have 
   \begin{equation}
       \limsup_{k\rightarrow \infty}\frac{\dim \left(f_*\mathcal{O}(kL)_{y}/\mathfrak{m}_yf_*\mathcal{O}(kL)_{y}\right) }{k^{\dim Z-\dim Y}}>0.
   \end{equation}
\end{lemma}

\begin{proof}
Let $T \in c_1(L) + f^*(\theta)$ be a closed positive $(1,1)$-current with local $\partial\bar\partial$-potentials, and suppose that $T \geq \alpha$ for some Kähler form $\alpha$ on $Z$. 
Take a resolution of singularities $\pi \colon Z' \to Z$ with $Z'$ a smooth compact Kähler manifold. 
Since $[\pi^*\alpha]$ is nef and big on $Z'$, it follows from \cite{demailly-paun} that $[\pi^*\alpha]$ admits a Kähler current, and hence $[\pi^*T]$ also admits a Kähler current. 
By Demailly’s regularization theorem on $Z'$, we can obtain a Kähler current in $[\pi^*T]$ with analytic singularities. 
Therefore, in the class $c_1(L) + f^*(\theta)$ we obtain a Kähler current $T'$ whose singular locus is contained in an analytic subvariety.

Let $y$ be a general smooth point of $Y$, and let $U$ be a small Stein neighborhood of $y$ such that $\theta$ is $\partial\bar\partial$-exact on $U$. 
Pick a point $z \in f^{-1}(y)$ such that both $Z$ and $T'$ are smooth in a neighborhood of $z$. 
Note that $T'$ corresponds to the curvature form of a Hermitian metric defined on the $\mathbb{Q}$-line bundle $L$. 
By a standard application of Hörmander’s $L^2$-estimate—prescribing the jets at $z$ of a local holomorphic section, cutting-off, and then applying the $L^2$-method with weights having logarithmic pole at $z$ on $f^{-1}(U)$—we can obtain the desired estimate.
\end{proof}

\item We will use the following consequence of Demailly’s regularization theorem~\cite{demailly1992}, together with resolution of singularities and the principalization of ideal sheaves~\cite{Hironaka1, wlodarczyk2008}, without explicitly mentioning it later.  
Let \( Z \) be a compact normal Kähler space and \( \alpha \in H^{1,1}_{\BC}(Z) \) a big class.  
Then there exists a projective resolution \( \nu: Z' \to Z \), a Kähler class \( \omega' \), and an effective \( \mathbb{Q} \)-divisor \( D' \) on \( Z' \) such that 
\begin{equation}
    \nu^*(\alpha) = \omega' + D' \quad \text{in } H^{1,1}_{\BC}(Z').
\end{equation}
Indeed, let \( T \in \alpha \) be a Kähler current with local \( \partial \bar{\partial} \)-potentials.  
Take a projective resolution of singularities \( \pi: \widetilde{Z} \to Z \), where \( \widetilde{Z} \) is a smooth Kähler manifold.  
Since there exists a Kähler form \( \omega \) on \( Z \) such that \( \pi^*T \ge \pi^*\omega \) as currents on \( \widetilde{Z} \), and since \( [\pi^*\omega] \) is nef and big on \( \widetilde{Z} \), it follows from~\cite{demailly-paun} that \( [\pi^*T] \) admits a Kähler current.
Applying Demailly’s regularization theorem to \( \widetilde{Z} \), and after possibly taking a further resolution, we obtain the desired decomposition.
\item We will frequently use the negativity lemma in the complex analytic setting, proved in \cite[Lemma~1.3]{wang2021iitaka}, which generalizes \cite[Lemma~3.39]{KM}, without giving an explicit reference each time.  
We remark that although the statement in \cite[Lemma~1.3]{wang2021iitaka} is formulated for Cartier divisors, the conclusion also holds for \( \mathbb{R} \)-Cartier divisors by a standard approximation argument, using the existence of an exceptional \( \mathbb{Q} \) divisor \( A \) such that \( -A \) is relatively ample after passing to a suitable blow-up.

\item We will use the following result without explicitly referring to it later.  
It is likely well known to experts; see, for example, \cite[Corollary~2.32]{Das_HP2024} and \cite[Lemma~4.3]{hacon-paun}.  
Since we could not find a detailed proof in the literature, we include one here for completeness.
\begin{lemma}\label{lem--adding exceptional no effect}
   Let \( f: Z \rightarrow Y \) be a fibration between compact normal Kähler spaces, and let \( \alpha \in H^{1,1}_{\BC}(Y) \).  
Suppose there exists an $\mathbb R$-Cartier divisor \( E \) on \( Z \) such that \( \mathrm{codim}_Y f(\operatorname{Supp} E) \ge 2 \) and that 
\(
f^*(\alpha) + [E] \in H^{1,1}_{\BC}(Z)
\)
is a pseudoeffective class on \( Z \).  
Then \( \alpha \) is pseudoeffective on \( Y \).
\end{lemma}
\begin{proof}
    Let $F=f(\supp(E))$ which is a closed analytic subset of codimension at least 2 in $Y$. Fix an open cover $\{U_i\}$ of $Y$ such that the class $\alpha$ admits a smooth representative $\theta$ with smooth local potentials on each $U_i$. Let $T$ the positive $(1,1)$-current with local psh potentials in the class $f^*(\alpha)+[E]$. Then there exists a distribution $\varphi$ on  $f^{-1}(Y\setminus F)$, such that as a distribution on $f^{-1}(Y\setminus F)$ we have
    \begin{equation}
        T=f^*(\theta)+\ii\p\pp \varphi.
    \end{equation}Since $T$ admits local psh potentials, then using \cite[Lemma 4.6.1]{boucksom-guedj-book}, we can obtain that there exist psh function $\varphi_i$ on each $f^{-1}(U_i\setminus F)$ such that as currents on $f^{-1}(U_i\setminus F)$ we have  
    \begin{equation}
        T=\ii\p\pp \varphi_i.
    \end{equation}Since for each $y\in Y$, the fiber $f^{-1}(y)$ is connected. Then by the psh property of   $\varphi_i$, we obtain that its restriction to each fiber is a constant (could be $-\infty$), i.e. there exists a function $\psi_i:U\setminus F\rightarrow [-\infty, \infty)$ such that $\varphi_i=f^*(\psi_i)$. Since $\varphi_i$ is psh, we know that $\psi_i$ is upper-semi-continuous, locally bounded above and locally integrable. By \cite[Theorem 1.7]{demailly1985-weaklypsh}, we know that to show $\psi_i$ is a psh function, it is enough to show that as a current 
    \begin{equation}
        \ii\p\pp \psi_i\geq 0.
    \end{equation}This follows form the psh property of $\varphi_i$ and the Fubini theorem, since after taking resolution of $Z$, we can find a K\"ahler form $\omega$ on $Z$ such that for a general point $y\in Y$, we have $\int_{Z_y}\omega^{\dim Z-\dim Y}=1$. Since $F$ has codimension at least 2, we know that $\psi_i$ extends as a psh function on $U_i$ and hence globally it defines a positive (1,1)-current in the class $\alpha$. Therefore we conclude that $\alpha$ is  pseudo-effective.
\end{proof}

\end{enumerate}

\subsection{Extremal contractions and $\mathbb Q$-factoriality}
Let \( f \colon Z \to Y \) be a projective fibration between compact normal Kähler spaces, and suppose that \( (Z, B) \) is a klt pair.
The relative cone and contraction theorems for \( K_Z + B \)-negative extremal rays have been established in \cite[Theorem~4.12]{nakayama1987}; see also \cite[Theorem~1.2]{fujino-cone} and \cite[Theorem~2.44]{Das_HP2024}.  
Let \( R \) be a \( K_Z + B \)-negative extremal ray in \( \overline{\mathrm{NE}}(Z/Y) \), and let \( \varphi \colon Z \to Z_1 \) be the associated extremal contraction.   By the construction in \cite{fujino-cone, fujino2022, Das_HP2024}, each \( Z_1 \) admits a projective fibration over \( Y \).  
Hence, by \cite[Lemma~4.4]{fujiki1978}, \(Z_1 \) is a compact normal Kähler space. 
In what follows, we will use this fact without explicitly mentioning it.

 If \( \varphi \) is a flipping contraction, its flip exists by \cite[Theorem~1.14]{fujino2022}.  
Let \( \varphi^{+} : Z^{+} \rightarrow Z_1 \) denote the flip, and let  
\( \phi : Z \dashrightarrow Z^{+} \) be the induced small bimeromorphic map and $B^+:=\phi_*B$.  
We remark that the existence of the flip is established via the finite generation of the \( \mathcal{O}_{Z_1} \)-algebra
\begin{equation}
    \bigoplus_{m \in \mathbb{N}} \varphi_{*} \mathcal{O}_{Z}\!\left(\left\lfloor m\left(K_{Z} + B\right)\right\rfloor\right),
\end{equation}
which can be verified locally on \( Z_1 \); see \cite[Theorem~1.8]{fujino2022}.
If $\dim Z_1<\dim Z$, then we will call it a Mori-Fano fibration. i.e. 
\begin{itemize}
     \item  $\dim (Z_1)< \dim (Z)$ and the relative Picard number $\rho(Z/Z_1)=1$.
    \item $(Z,B)$ is a klt pair and $-(K_{Z}+B)$ is $\varphi$-ample;
\end{itemize}

The following lemma is well known to experts; see \cite[Lemma~2.5]{dh-3fold}.  
For the reader’s convenience and for later reference, we state it here together with a proof.
\begin{lemma}\label{q-factroial is preserved}
Let \( (Z, B) \) be a klt pair, and let 
\( \varphi \) be a \( K_Z + B \)-negative extremal contraction with relative Picard number \( 1 \).  
Then:
\begin{enumerate}
     \item If \( \varphi \) is a divisorial contraction, then \( Z_1 \) is \( \mathbb{Q} \)-factorial if $Z$ is $\mathbb Q$-factorial and $Z_1$ is strongly $\mathbb Q$-factorial if $Z$ is strongly $\mathbb Q$-factorial;
    \item If \( \varphi \) is a flipping contraction, then the flip \( Z^{+} \) is \( \mathbb{Q} \)-factorial  if $Z$ is $\mathbb Q$-factorial and $Z_1$ is strongly $\mathbb Q$-factorial if $Z$ is strongly $\mathbb Q$-factorial;
    \item If \( \varphi \) is a Mori–Fano fibration, then \( Z_1 \) is strongly \( \mathbb{Q} \)-factorial if \( Z \) is strongly \( \mathbb{Q} \)-factorial.
\end{enumerate}
\end{lemma}
\begin{proof}
In the case where \( Z \) is strongly \( \mathbb{Q} \)-factorial and \( \varphi \) is a divisorial or flipping contraction, the statement is proved in \cite[Lemma]{dh-3fold}.  
Combining that argument with the standard techniques from the algebraic case 
(see \cite[Corollary~3.18 and Proposition~3.37]{KM}), the statement for the case where \( Z \) is \( \mathbb{Q} \)-factorial also follows.  
We now consider the case where \( \varphi \) is a Mori–Fano fibration.

Let $\mathcal{F}$ be a divisorial sheaf on $Z_1$, i.e., a rank-one reflexive coherent sheaf.
Then \( \mathcal{F} \) is locally free on \( Z_1^{\mathrm{reg}} \).  
Consider the reflexive pullback \( (\varphi^*(\mathcal{F}))^{**} \), which coincides with \( \varphi^*(\mathcal{F}) \) on \( \varphi^{-1}(Z_1^{\mathrm{reg}}) \).  
Since \( Z \) is strongly \( \mathbb{Q} \)-factorial, \( (\varphi^*(\mathcal{F}))^{**} \) is a \( \mathbb{Q} \)-line bundle on \( Z \). Let \( C \) be an irreducible curve whose class generates the $K_X+B$-negative extremal ray, and we may assume that \( \varphi(C)=pt \subset Z_1^{\mathrm{reg}} \).  
Then
\begin{equation}
    C \cdot (\varphi^*(\mathcal{F}))^{**} = 0.
\end{equation}
Hence, by the property of extremal contractions in \cite[Theorem 4.12]{nakayama1987}, there exists a line bundle \( \mathcal{L} \) on \( Z_1 \) such that 
\begin{equation}
    (\varphi^*(\mathcal{F})^{\otimes m})^{**} \simeq \varphi^*(\mathcal{L}).
\end{equation}
Since \( (\varphi^*(\mathcal{F})^{\otimes m})^{**} \) coincides with \( \varphi^*(\mathcal{F}^{\otimes m}) \) on \( \varphi^{-1}(Z_1^{\mathrm{reg}}) \) and $\mathcal{F}^{\otimes m}$ coincides with $(\mathcal{F}^{\otimes m})^{**}$ on $Z_1^{\mathrm{reg}}$,
the projection formula implies that 
\[
\mathcal{L} \simeq (\mathcal{F}^{\otimes m})^{**} \quad \text{on } Z_1^{\mathrm{reg}}.
\]
As both \( \mathcal{L} \) and \( (\mathcal{F}^{\otimes m})^{**} \) are reflexive sheaves and the complement of $Z_1^{\mathrm{reg}}$ in $Z_1$ has codimension at least 2, 
it follows that \( (\mathcal{F}^{\otimes m})^{**} \) is globally isomorphic to \( \mathcal{L} \) on \( Z_1 \); in particular, it is a line bundle.
\end{proof}

\begin{remark}
 Note that for normal complex analytic spaces the canonical sheaf need not be represented by a $\mathbb{Q}$-Weil divisor and in our convention, $\mathbb{Q}$-factoriality also requires that the canonical sheaf be a $\mathbb{Q}$-line bundle. With this in mind, for a Mori–Fano fibration $\varphi\colon Z\to Z_1$, it seems not clear that $Z_1$ is $\mathbb{Q}$-factorial under the sole assumption that $Z$ is $\mathbb{Q}$-factorial.
\end{remark}

\subsection{Pushing-forward some classes in $H^{1,1}_{\BC}$}\label{sec-definiton of pushforward}
In general, pushing forward a class in \( H^{1,1}_{\BC} \) is a subtle matter.  
For our purposes, we define the pushforward only for certain special classes and only in the case where we have a \( K_Z + B \)-negative extremal contraction with relative Picard number one.  
In the bimeromorphic case, the pushforward is relatively straightforward and follows from Section~\ref{sec:preliminary}–(4); see, for instance, \cite[Lemma~2.14]{dhy}.  
In the case of a Mori–Fano fibration, we rely on the canonical bundle formula established in \cite{hacon-paun}.

Let \( f \colon Z \to Y \) be a projective fibration between compact normal Kähler spaces.  
We always assume that \( Z \) is \( \mathbb{Q} \)-factorial and that \( (Z,B) \) is a klt pair.  
Let \( \varphi \colon Z \to Z_1 \) be the extremal contraction associated to a \( K_Z + B \)-negative extremal ray \( R \in \overline{\mathrm{NE}}(Z/Y) \).  
We know that \( \dim N^1(Z/Z_1) = 1 \), and that the following exact sequence holds:
\begin{equation}
    0 \longrightarrow N^1(Z_1/Y) \xrightarrow{\varphi^*} N^1(Z/Y) \longrightarrow N^1(Z/Z_1) \longrightarrow 0.
\end{equation}
Since \( K_Z + B \) is relatively ample over \( Z_1 \) and \( \dim N^1(Z/Z_1) = 1 \), for any \( \alpha \in N^1(Z/Y) \), there exists a unique real number \( \lambda_\alpha \in \mathbb{R} \) such that  
\begin{equation}\label{eq-definition of lambda}
    \alpha + \lambda_\alpha (K_Z + B) \in \operatorname{Im}(\varphi^*).
\end{equation}For later applications, we mainly interested in those $\alpha\in H^{1,1}_{\BC}(Z)$ with  $\lambda_{\alpha}>0$.

We further assume that the general fibers of \( f \colon Z \to Y \) are rationally connected, and throughout the remainder of this subsection we work under this assumption.  
As mentioned in Section~\ref{sec:preliminary}, this implies \eqref{eq--divisor equal H2}, which will play an important role in what follows.  
Let \( \alpha \in H^{1,1}_{\BC}(Z) \).  
Using \eqref{eq--divisor equal H2}, we let \( \lambda_{\alpha} \) denote, as before, the real number such that \eqref{eq-definition of lambda} holds, which means there exists \( \beta \in H^{1,1}_{\BC}(Z_1) \) such that  
\begin{equation}\label{eq-class is the pull-back class}
    \alpha + \lambda_{\alpha}(K_Z + B) = \varphi^*\beta \quad \text{in } H^{1,1}_{\BC}(Z).
\end{equation}
Such a class \( \beta \) is unique, as noted in Section~\ref{sec:preliminary}–(3).

We first define the pushforward of a class \( \alpha \) when \( \varphi \) is bimeromorphic.  

\begin{definition}\label{def-pushforward of numerical class} Let $B_{1}:=\varphi_*(B)$ and $B^+:=\phi_*(B)$.

  \begin{enumerate}
      \item If $\varphi$ is a divisorial contraction, then  we define 
      \begin{equation}
          \varphi_*(\alpha):=\beta-\lambda_{\alpha} (K_{Z_1}+B_{1})\in H^{1,1}_{\BC}(Z_1).
      \end{equation}
      \item If $\varphi$ is a flipping contraction, we define
      \begin{equation}
          \phi_*(\alpha):=(\varphi^+)^*(\beta)-\lambda_{\alpha} (K_{Z^+}+B^{+}) \in H^{1,1}_{\BC}(Z^+).
      \end{equation}
  \end{enumerate}
\end{definition}

We make a few remarks about this  definition. 
\begin{itemize}
    \item Note that in the above definition, we implicitly used the fact that $(Z_1,B_1)$ a $\mathbb Q$-factorial klt pair when $\varphi$ is divisorial contraction and $(Z^+,B^+)$ is a $\mathbb Q$-factorial klt pair when $\varphi$ is a flipping contraction; see Lemma \ref{q-factroial is preserved}.
    \item  Clearly both $\varphi_*$ and $\phi_*$ are linear maps and $\varphi_*$ is surjective and $\phi_*$ is bijective. 
        \item Without assuming \eqref{eq--divisor equal H2}, we can still define pushing-forward of classes in $N^1(Z/Y)$.
    \item  The definition is chosen to make the following if $K_Z+B+\alpha$ is nef over $Y$ and $\varphi$-trivial, then $K_{Z_1}+B_1+\varphi_*(\alpha)$ and $K_{Z^+}+B^++\phi_*(\alpha)$ is nef over $Y$ respectively.
    \item This definition is compatible with the notion of pushing-forward Weil divisors, i.e. if a class $\alpha$ is represented by a $\mathbb R$-divisor $W$, then $\varphi_*(\alpha)$ (respectively $\phi_*(\alpha)$) coincides with the numerical class definied by the $\mathbb R$ divisor $\varphi_*(W)$ (respectively $\phi_*(W)$). (Note that we assumed that $Z$ is $\mathbb Q$-factorial.)
\end{itemize}

Now consider the case where $\varphi\colon Z\to Z_1$ is the Mori–Fano fibration associated to a $K_Z+B$-negative extremal ray. We restrict to classes $\alpha\in H^{1,1}_{\BC}(Z)$ with $\lambda_{\alpha}>0$. By a straightforward scaling argument, it suffices to define $\varphi_{*}(\alpha)$ in the normalized case $\lambda_{\alpha}=1$, since the general case then follows by homogeneity. In this situation, our definition of $\varphi_{*}(\alpha)\in H^{1,1}_{\BC}(Z_1)$ relies on the canonical bundle formula for generalized pairs established by Hacon–Păun \cite{hacon-paun}. 
We refer to \cite[Section 2]{dhy} for definitions and properties of generalized pairs. In this paper, we will abuse the notation using the following definition.
\begin{definition}\label{def-numerical generalized pair}
    Let $Z$ be a  $\mathbb Q$-factorial compact normal K\"ahler space, $B$ an effective $\mathbb R$-divisor and $\alpha \in H^{1,1}_{\BC}(Z)$. We say $(Z,B+\alpha)$ is a generalized pair if there exists a generalized pair $(Z, B+\boldsymbol{\beta})$ in the sense of \cite[Definition 2.7]{dhy} with $\boldsymbol{\beta}$ a positive closed b-(1,1) current such that 
\begin{equation}
    \alpha=[\boldsymbol{\beta}_Z] \text{ in } H^{1,1}_{\BC}(Z).
\end{equation}
\end{definition}
 Note that without the \( \mathbb{Q} \)-factorial assumption, the notion of \( [\boldsymbol{\beta}_Z] \) does not make sense, since the positive \((1,1)\)-current \( \boldsymbol{\beta}_Z \) may fail to admit local potentials. Under the assumption that \( Z \) is \( \mathbb{Q} \)-factorial, this can be understood as follows.  
Take a log resolution \( \nu: (Z', B') \to (Z, B) \). Then there exists a positive \((1,1)\)-current \( \boldsymbol{\beta}_{Z'} \) on \( Z' \) with local potentials such that \( [\boldsymbol{\beta}_{Z'}] \) is nef and 
\[
[K_{Z'} + B' + \boldsymbol{\beta}_{Z'}] = \nu^*(\gamma)
\]
for some \( \gamma \in H^{1,1}_{\BC}(Z) \).  
Since \( Z \) is \( \mathbb{Q} \)-factorial, we can write
\begin{equation}
    [\boldsymbol{\beta}_{Z'}] + K_{Z'} + B' - \nu^*(K_Z + B) = \nu^*(\gamma - (K_Z + B)).
\end{equation}
By the negativity lemma, we know that \( K_{Z'} + B' - \nu^*(K_Z + B) = E \) is an  exceptional effective $\mathbb R$-divisor.  
Pushing forward the positive \((1,1)\)-current \( \boldsymbol{\beta}_{Z'} + E \) to \( Z \) then yields $\boldsymbol{\beta}_{Z}$, a positive \((1,1)\)-current with local potentials in the class $\gamma-(K_Z+B)$, as discussed in \cite[Remark~2.6(ii)]{dhy}.

\begin{definition}\label{def--pushing forward for mori-fano fibration}
    Let $\varphi:Z\rightarrow Z_1$ be a Mori-Fano fibration. Given $\alpha\in H^{1,1}_{\BC}(Z)$. Suppose $Z$ is strongly $\mathbb Q$-factorial and  $(Z,B+\alpha)$ is a generalized klt and is $\varphi$-trivial. By \cite[Theorem 0.3]{hacon-paun}, there exists a generalized klt pair $(Z_1, B_{Z_1}+\boldsymbol{\beta}_{Z_1})$ such that $K_Z+B+\alpha=\varphi^*(K_{Z_1}+B_{Z_1}+\boldsymbol{\beta}_{Z_1})$. Then we define
    \begin{equation}
        \varphi_*(\alpha)=[\boldsymbol{\beta}_{Z_1}] \quad \text{ in } \quad H^{1,1}_{\BC}(Z_1).
    \end{equation}
\end{definition}

We make a few remarks about this  definition. 
\begin{itemize}
    \item We have used the fact that \( Z_1 \) is \( \mathbb{Q} \)-factorial; see Lemma \ref{q-factroial is preserved}. Recalling our notational convention, we thus obtain that \( (Z_1, B_{Z_1} + \varphi_*(\alpha)) \) forms a generalized klt pair. Moreover, by \cite[Lemma~2.13]{dhy}, the pair \( (Z_1, B_{Z_1}) \) is klt.
    \item  In the definition of generalized klt singularities \cite[Definition~2.9]{dhy}, one considers the \emph{generalized discrepancy}, which may differ from the usual discrepancy of a pair.  
Hence, the condition that \( (Z, B + \alpha) \) is a generalized klt pair is nontrivial and does not follow automatically from the assumption that \( (Z, B) \) is klt.  
However, by \cite[Remark~2.10]{dhy}, this holds when \( \alpha \) is nef.  
In our later applications involving the relative MMP, we will use this fact to ensure that the initial data forms a generalized klt pair.
\end{itemize}

\subsection{Bigness of boundary classes preserved after extremal contractions} We will  show in this section that the bigness of the boundary class is preserved under \( K_Z + B \)-negative extremal contractions of relative Picard number 1.
To simplify notation, when the extremal contraction is bimeromorphic, we use the same symbol \( \varphi: Z \dashrightarrow Z_1 \) to denote either a divisorial contraction or a flip associated with a \( K_Z + B \)-negative extremal ray.  
Whenever we refer to the pushforward of a class in \( H^{1,1}_{\BC}(Z) \) by \( \varphi \), it is always understood in the sense of Definitions~\ref{def-pushforward of numerical class} and~\ref{def--pushing forward for mori-fano fibration}.

We first note the following immediate consequence of the definition in the bimeromorphic case.

\begin{lemma}\label{lem--more effective before contraction}
    Suppose $\lambda_{\alpha}>0$, then there exists a common resolution $Z\xlongleftarrow[]{\nu}Z'\xlongrightarrow[]{\nu_1} Z_1$ such that for any $t\in [0,\lambda_{\alpha}^{-1}]$, there exists effective $\nu_1$-exceptional $\mathbb R$-divisor $E_t$ such that 
    \begin{equation}
        \nu^*(K_Z+B+t \alpha)=\nu_1^*(K_{Z_1}+B_{1}+t\varphi_*(\alpha))+E_t
    \end{equation}
\end{lemma}
\begin{proof}
  By the definition of the pushforward, there exists a common resolution of \( Z \) and \( Z_1 \),
\[
Z \xlongleftarrow[]{\nu} Z' \xlongrightarrow[]{\nu_1} Z_1,
\]
such that  
\begin{equation}
    \nu^*(K_{Z} + B + \lambda_{\alpha}^{-1} \alpha)
    = \nu_1^*\bigl(K_{Z_{1}} + B_{1} + \lambda_{\alpha}^{-1} \varphi_*(\alpha)\bigr).
\end{equation}
Hence, for any \( t \in [0, \lambda_{\alpha}^{-1}] \), we have
\begin{equation}
    \nu^*(K_{Z} + B + t \alpha)
    = \nu_1^*\bigl(K_{Z_{1}} + B_{1} + t \varphi_*(\alpha)\bigr)
      + (1 - \lambda_{\alpha} t)\bigl(\nu^*(K_{Z} + B) - \nu_1^*(K_{Z_{1}} + B_{1})\bigr).
\end{equation}
By the negativity lemma, the difference \( \nu^*(K_{Z} + B) - \nu_1^*(K_{Z_{1}} + B_{1}) \) is an effective $\nu_1$-exceptional $\mathbb Q$-divisor.
\end{proof}

\begin{lemma}\label{lem--bigness preserved under pushforward} Suppose $\lambda_{\alpha}>0$, then
    if $\alpha$ is a big class, then $\varphi_*(\alpha)$ is big class.
\end{lemma}
\begin{proof}
    When $\varphi$ is a divisorial contraction, to show $\varphi_*(\alpha)$ is a big class, it suffices to show that $\varphi^*(\varphi_*(\alpha))$ is a big class. By definition, we have 
    \begin{equation}
        \varphi^*(\varphi_*(\alpha))=\alpha+\lambda_{\alpha}((K_Z+B)-\varphi^*(K_{Z_1}+B_1)).
    \end{equation}By negativity lemma, we know that $(K_Z+B)-\varphi^*(K_{Z_1}+B_1)$ is a positive multiple of the exceptional divisor of $\varphi$. Since $\lambda_{\alpha}$ is positive, we obtain that  $\varphi^*(\varphi_*(\alpha))$ is a big class.

When $\varphi$ is a flip map, taking a common resolution $ Z'$ as before and let let $\pi$ and $\pi_1$ denotes the corresponding morphism to $Z$ and $Z_1$. We will show that $\pi_1^*(\varphi_*(\alpha))$ is a big class. By definition, we have
\begin{equation}
    \pi_1^*(\varphi_*(\alpha))=\pi^*(\alpha)+\lambda_{\alpha}(\pi^*(K_Z+B)-\pi_1^*(K_{Z_1}+B_1)).
\end{equation}Again by negativity lemma, we know that $\pi^*(K_Z+B)-\pi_1^*(K_{Z_1}+B_1)$ is an effective $\pi_1$-exceptional $\mathbb Q$-divisors. Since $\alpha$ is a big class, we conclude that $\pi_1^*(\varphi_*(\alpha))$ is a big class and hence $\varphi_*(\alpha) $ is a big class.
\end{proof}

The following result is a consequence of \cite[Theorem~6.2]{hacon-paun}; see also \cite{paun-takayama} and \cite[Theorem~5.2]{cao-horing}.  
It plays an important role in our later applications in Section~\ref{sec-structur theorem}, as it ensures that even after reaching a Mori–Fano fibration, the relative MMP can still be continued on the base with the scaling of a big class.
\begin{lemma}\label{lem-bigness if preserved under fano fibration}
 Let \( \varphi : Z \to Z_1 \) be a Mori–Fano fibration, where \( Z \) is strongly \( \mathbb{Q} \)-factorial, and let \( \alpha \in H^{1,1}_{\BC}(Z) \). 
Assume that \( (Z, B + \alpha) \) is a generalized klt pair and that \( K_Z + B + \alpha \) is \( \varphi \)-trivial.  
If the class \( B + \alpha \) is big on \( Z \), then the class \( B_{Z_1} + \varphi_*(\alpha) \) is big on \( Z_1 \).
\end{lemma}

\begin{proof}
    Taking a log resolution $\nu: Z_1'\rightarrow Z_1$ and then by eliminating of indeterminacy and taking further log resolutions, we can obtain the following commutative diagram of projective fibrations
    \begin{equation}
         \begin{tikzcd}
Z \arrow[d, "\varphi"'] & Z' \arrow[l, "\nu"'] \arrow[d, "\varphi'"] \\
Z_1 & Z_1' \arrow[l, "\nu_1"'].
\end{tikzcd}
    \end{equation}
  Since $(Z, B+\alpha)$ is a generalized pair, then by its definition, we know that we can write 
  \begin{equation}\label{eq--generalized klt relation}
      K_{Z'}+B'_{+}-B'_{-}+\alpha'=\nu^*(K_Z+B+\alpha)
  \end{equation}such that 
  \begin{itemize}
      \item $B'_{+}, B'_{-}$ are effective $\mathbb R$-divisors with no common components and have simple normal crossing support;
      \item  $\lfloor B'_{+}\rfloor=0$, i.e. coefficients of irreducible components in $B'_+$ is in $(0,1)$;
      \item $B'_-$ is $\nu$-exceptional;
      \item $\alpha'$ is a nef class.
  \end{itemize}

 Since \( B + \alpha \) is a big class, by taking further resolutions we may assume 
\(
\nu^*(B+\alpha) = \omega + F,
\)
where \( F \) is an effective \( \mathbb{Q} \)-divisor such that \( B'_+ + B'_- + F \) has simple normal crossing support and \( \omega \) is a Kähler class on $Z'$.  
Then combining with \eqref{eq--generalized klt relation}, we know that there exist effective $\nu$-exceptional \( \mathbb{Q} \)-divisors \( E_+ \) and \( E_- \) such that 
\begin{equation}
    \alpha' = -B'_+ + B'_- + \omega + F + E_+ - E_-.
\end{equation}
For \( 0 < \epsilon \ll 1 \), we can write
\begin{equation}\label{eq-decomposition on Z'}
    K_{Z'} + \bigl((1-\epsilon)B'_+ + \epsilon F + \epsilon B'_- + \epsilon E_+\bigr)
    + \bigl((1-\epsilon)\alpha' + \epsilon \omega\bigr) - \epsilon E_- 
    = \nu^*(K_Z + B + \alpha).
\end{equation}
Let 
\[
    D := (1-\epsilon)B'_+ + \epsilon F + \epsilon B'_- + \epsilon E_+\geq 0.
\]
We fix \( \epsilon\in \mathbb Q_{>0} \) sufficiently small so that 
\begin{equation}\label{eq--klt condition sec 2}
    \lfloor D \rfloor = 0.
\end{equation}
On \( Z_1' \), there exist an $\nu_1$-exceptional $\mathbb Q$-divisor $E_1$ such that
\begin{equation}\label{eq-decomposition on Z-1'}
    K_{Z_1'} = \nu_1^*(K_{Z_1}) + E_1.
\end{equation}
We can choose a Kähler class \( \omega_1 \) on \( Z_1 \) such that 
\[
\omega'=(1-\epsilon)\alpha' + \epsilon \omega - (\nu_1 \circ \varphi')^*(\omega_1)
\]
remains a Kähler class on \( Z' \).  
Then, by the definition of the pushforward of a class given in Definition~\ref{def--pushing forward for mori-fano fibration}, and combining the commutative diagram with \eqref{eq-decomposition on Z'} and \eqref{eq-decomposition on Z-1'}, we obtain
\begin{equation}\label{eq--to apply hacon-paun}
    K_{Z'} - (\varphi')^*(K_{Z_1'}) 
    + D + \omega'
    - \epsilon E_-= (\varphi')^*\bigl(\nu_1^*(B_{Z_1} + \varphi_*(\alpha) - \omega_1) - E_1\bigr).
\end{equation}
We can perturb the coefficients in \( D \)  to rational numbers, 
absorbing the small error term into the Kähler class 
\(
(1-\epsilon)\alpha' + \epsilon \omega - (\nu_1 \circ \varphi')^*(\omega_1),
\)
so that \( D \)  may be assumed to be effective \( \mathbb{Q} \)-divisors with simple normal crossing support.  
Let \( L \) denote the \( \mathbb{Q} \)-line bundle defined by the effective \( \mathbb{Q} \)-divisor \( \epsilon E_- \).  
Then, for sufficiently large and divisible \( m \) and for a general point \( p \in Z_1' \), we have 
\begin{equation}\label{eq-nonvanishing of sections}
    H^0(Z_p',\, mL|_{Z'_p}) \neq 0.
\end{equation}
Hence, by \eqref{eq--klt condition sec 2} and \eqref{eq-nonvanishing of sections}, the decomposition in \eqref{eq--to apply hacon-paun} satisfies the conditions of \cite[Theorem~6.2]{hacon-paun}, 
and applying this result we obtain that the following class is pseudo-effective on \( Z' \):
\begin{equation}
    (\varphi')^*\bigl(\nu_1^*(B_{Z_1} + \varphi_*(\alpha) - \omega_1) - E_1\bigr) + \epsilon E_-.
\end{equation}

Since \( E_- \) is \( \nu \)-exceptional, 
by Lemma~\ref{lem--adding exceptional no effect} we deduce that 
\[
    \varphi^*(B_{Z_1} + \varphi_*(\alpha) - \omega_1) - \nu_*((\varphi')^*E_1)
\]
is a pseudo-effective class on \( Z \).  
As \( E_1 \) is \( \nu_1 \)-exceptional, the image 
\( \varphi(\operatorname{Supp}(\nu_*((\varphi')^*E_1))) \) 
has codimension at least two in \( Z_1 \). Then by Lemma~\ref{lem--adding exceptional no effect} again, 
the class \( B_{Z_1} + \varphi_*(\alpha) - \omega_1 \) is pseudo-effective on \( Z_1 \).  
Hence \( B_{Z_1} + \varphi_*(\alpha) \) is a big class.
\end{proof}

\subsection{Relative MMP with scaling of a relatively big $\mathbb R$-line bundle}Recall that by a relative \( \mathbb{R} \)-line bundle we mean an element of \( N^1(Z/Y) \); 
that is, we always identify it with its numerical equivalence class in \( N^1(Z/Y) \).  
In the next section, we will need to run a relative MMP with scaling of a relatively big \( \mathbb{R} \)-line bundle \( \alpha \in N^1(Z/Y) \).  
Since most results in the literature are formulated for the scaling of \emph{divisors}, 
we provide a brief justification for carrying out the relative MMP with scaling of a relatively big \( \mathbb{R} \)-line bundle; 
see also \cite[Theorem~2.10]{claudon-horing}.  
For our purposes, we are only concerned with the case where \( K_Z + B \) is not relatively pseudo-effective, 
and we therefore state the following specialized version.
\begin{theorem}[\cite{nakayama1987,fujino-cone,fujino2022,Das_HP2024}]\label{thm-relative mmp}

Let \( f: Z \rightarrow Y \) be a projective fibration between normal compact Kähler spaces. 
Suppose that \( Z \) is \(\mathbb{Q}\)-factorial and that \((Z,B)\) is a klt pair. 
Let \( \alpha \) be a relatively big \(\mathbb{R}\)-line bundle on \( Z \) such that \( K_Z + B +  \alpha \) is $f$-nef. Suppose $K_Z+B$ is not $f$-pseudo-effective, then we can then run the \((K_Z + B)\)-MMP with scaling of \(\alpha\) over \(Y\) to get a Mori-Fano fibration. 
More precisely, there exists a finite sequence of flips and divisorial contractions $\{\varphi_{i}\}_{i=0}^{m-1}$ which are induced by contracting $K_{Z_i}+B_i$-negative extremal rays
\[
(Z,B) = (Z_0, B_0)  \xdashrightarrow{\varphi_0}
(Z_1, B_1)
\xdashrightarrow{\varphi_1} \cdots 
\xdashrightarrow{\varphi_{m-1}} (Z_m, B_m)\xrightarrow{\psi} W,
\]
where \( B_{i+1} := (\varphi_i)_* B_i \) for every \( i \ge 0 \). Moreover let $\alpha_0=\alpha$ and $\alpha_{i+1}=(\varphi_i)_*(\alpha_i)$, then there exists positive decreasing numbers
  $  1\geq \lambda_0\geq \lambda_1\cdots \geq \lambda_m>0
$ such that 
\begin{itemize}
    \item $\lambda_i=\inf\{\lambda\in [0, \lambda_{i-1}] \mid K_{Z_{i}}+B_i+\lambda \alpha_i \text{ is nef over } Y\}$ with the convention that $\lambda_{-1}=1$.
    \item $\psi$ is a Mori-Fano fibration. 
\end{itemize}
\end{theorem}
\begin{proof}
   Suppose that the sequences \( \{(Z_i, B_i)\}_{i=0}^{k} \), \( \{\varphi_i\}_{i=0}^{k-1} \), and \( \{\lambda_i\}_{i=0}^{k-1} \) have been constructed satisfying the properties stated in the theorem.  
We now construct \( Z_{k+1} \), \( \varphi_k \), and \( \lambda_k \).  
We may assume \( \dim Z_k = \dim Z_0 \); otherwise, we have already obtained a Mori–Fano fibration \( \varphi_{k-1} \colon Z_{k-1} \to Z_k \), and the process terminates.  
Define
\begin{equation}
    \lambda_k = \inf \left\{ t \in [0, \lambda_{k-1}] \,\middle|\, K_{Z_k} + B_k + t \alpha_k \text{ is nef over } Y \right\}.
\end{equation} By \cite[Theorem~1.2~(7)]{fujino-cone},  there exists a \( (K_{Z_k} + B_k) \)-negative extremal ray \( R \subset \overline{\mathrm{NE}}(Z_k/Y) \) such that  
\begin{equation}
    R \cdot (K_{Z_k} + B_k + \lambda_k \alpha_k) = 0.
\end{equation}
We then obtain an extremal contraction over \( Y \), defined by \( R \) as guaranteed by \cite[Theorem~1.2]{fujino-cone},  
\[
\varphi' \colon Z_k \longrightarrow Z'.
\]

If $\dim Z'<\dim Z_k$, then we let $W=Z'$ and $\psi=\varphi'$. Then we get the desired Mori-Fano fibration and we stop.

If $\varphi'$ is a divisorial contraction, then we let $Z_{k+1}=Z'$ $\varphi_k=\varphi'$,  $B_{k+1}:=(\varphi')_*(B_k)$ and $\alpha_{k+1}=(\varphi_k)_*(\alpha_k)$ as defined in Definition \ref{def-pushforward of numerical class}.

 If $\varphi'$ is a flip contraction,  then  $K_{Z_k}+B_k$-flips exists by \cite[Theorem 1.14]{fujino2022}. Note that this  is established via the finite generation of the $\mathcal O_{Z'}$-algebra \begin{equation}
\bigoplus_{m \in \mathbb{N}} \varphi'_* \mathcal{O}\left(\left\lfloor m\left(K_{Z_k}+B_k\right)\right\rfloor\right)
\end{equation}which can be proved locally \cite[Theorem 1.8]{fujino2022}. Then we let $\varphi_{k}:Z_k\dashrightarrow Z_{k+1}$ be the flip of $\varphi'$ and let $B_{k+1}:=(\varphi_k)_*B_k$ and $\alpha_{k+1}=(\varphi_k)_*(\alpha_k)$.

We now show that, after finitely many steps of divisorial contractions and flips, the process terminates with a Mori–Fano fibration.  
We argue by contradiction.  
Suppose there exists an infinite sequence of bimeromorphic maps \( \varphi_i \) satisfying the conditions of the theorem.  
We first show that, as a consequence of the relative non-pseudoeffectivity of \( K_Z + B \), we have  
\begin{equation}
    \lim_{i \to \infty} \lambda_i > 0.
\end{equation} 
Repeating Lemma \ref{lem--more effective before contraction} for each divisorial contraction and flip, we can obtain that  for each \( i \), \( K_Z + B + \lambda_i \alpha \) is \( f \)-pseudoeffective.  
If \( \lim_{i \to \infty} \lambda_i = 0 \), then \( K_Z + B \) would be \( f \)-pseudoeffective, contradicting our assumption.

By Lemma~\ref{lem--bigness preserved under pushforward}, each \( \alpha_i \) is relatively big over \( Y \).  
Cover \( Y \) by finitely many open sets \( \{U_j\} \), each is Stein and restrict the morphisms over \( U_j \).  
Since each \( U_j \) is Stein, the numerical class \( \alpha \) restricted to \( f^{-1}(U_j) \) can be represented by a relatively big \( \mathbb{R} \)-divisor.  
Hence, by the finiteness of models \cite[Theorem~E]{fujino2022}, for each \( j \), the restrictions \( Z_i|_{U_j} \) belong to only finitely many isomorphism classes.  
Moreover, by the negativity lemma, the discrepancy of any divisor over \( Z \) whose center lies in the flipped locus must \emph{strictly} increase.  
Thus, for sufficiently large \( i \), if the flipped locus were nonempty, we would reach a contradiction.  
On the other hand, if there were infinitely many \( \varphi_i \), the flipped locus could not be nonempty.  
We therefore conclude that only finitely many \( \varphi_i \) can be bimeromorphic.
\end{proof}

The following property is well known to experts and is the analytic analogue of the fact that discrepancies can only increase along the MMP.

\begin{lemma}\label{lem-generalized klt is preserved}
In the same setting as Theorem~\ref{thm-relative mmp}, suppose further that  \( \alpha \in H^{1,1}_{\BC}(Z) \) such that \( (Z, B + \alpha) \) is a generalized klt pair.  
Then each \( (Z_i, B_i + \lambda_i\alpha_i) \) is a generalized klt pair.
\end{lemma}

\begin{proof}
By examining each divisorial contraction and flip, it suffices to show that \( (Z_1, B_1 + \lambda_1 \alpha_1) \) is a generalized klt pair.  
By Lemma~\ref{lem--more effective before contraction} and the definition of generalized klt pairs in \cite[Definition~2.9]{dhy}, it is enough to prove that \( (Z, B + t \alpha) \) is a generalized klt pair for every \( t \in [0,1] \).

For this, recall that there exists a bimeromorphic morphism
\(
\nu : Z' \longrightarrow Z
\)
from a Kähler manifold \( Z' \), together with an \( \mathbb{R} \)-divisor \( B' \) with simple normal crossing support such that \( B = \nu_* B' \ge 0 \), and a nef class \( \alpha' \) on \( Z' \) satisfying
\begin{equation}
    K_{Z'} + B' + \alpha' = \nu^*(K_Z + B + \alpha),
    \qquad \lfloor B' \rfloor \leq  0.
\end{equation}
Then for any \( t \in [0,1] \), we can write
\begin{equation}
    K_{Z'} + B' - (1-t)E + t\alpha' = \nu^*(K_Z + B + t\alpha),
\end{equation}
where \( E \) is an exceptional \( \mathbb{R} \)-divisor given by
\[
E = -\nu^*(K_Z + B) + (K_{Z'} + B') = -\alpha' + \nu^*(\alpha).
\]
Since \( \alpha' \) is nef, the negativity lemma implies \( E \ge 0 \).  
Moreover, as \( \lfloor B' \rfloor \leq 0 \), we have
\begin{equation}
    \lfloor B' - (1-t)E \rfloor \leq 0.
\end{equation}
Hence \( (Z, B + t\alpha) \) is a generalized klt pair for all \( t \in [0,1] \).
\end{proof}

\section{A weak transcendental base-point freeness result}\label{sec-structur theorem}
In this section, we first prove Theorem~\ref{thm-weak basepointfree}, which provides a weaker bimeromorphic version of the transcendental base-point-freeness theorem.  
We then state a more general version for generalized klt pairs (see Theorem~\ref{thm--generalized klt version}) and discuss some consequences of these results.  
In particular, Proposition~\ref{prop-dim of base coincide with numerical} shows that the dimension of the base \( Y \) in Theorem~\ref{thm-weak basepointfree} coincides with the numerical dimension of the adjoint class \( K_X + \alpha \).  Corollary \ref{cor-numerical dim at most 3 is done} shows that the transcendental base-point freeness holds whenever the adjoint class has numerical dimension at most 3.
Moreover, in Proposition~\ref{prop-generalization of das-hacon 4.1}, we extend \cite[Theorem~4.1]{das-hacon2024} to the Kähler setting. 


\subsection{Proof of Theorem \ref{thm-weak basepointfree}}
We firstly prove the following result using a similar argument as in Lemma \ref{lem-bigness if preserved under fano fibration} relying on the pseudo-effectivity of the twisted relative canonical class proved in \cite[Theorem 6.2]{hacon-paun}.

\begin{lemma}\label{lem-not relatively psef}
    Let $f:Z\rightarrow Y$ be a projective fibration between compact normal  K\"ahler spaces with general fibers of $f$ being rationally connected. Suppose $(Z,B)$ is a $\mathbb Q$-factorial klt pair and $Y$ has canonical singularity. If $K_{Z}+B-f^*(K_Y)$ is not pseudo-effective, then $K_{Z}+B$ is not $f$-pseudo-effective. 
\end{lemma}
\begin{proof}
   First, recall that since the general fibers of $f\colon Z\to Y$ are rationally connected, we have (see Section~\ref{sec:preliminary}, item (7)) the identification
\[
N^{1}(Z/Y)\;\cong\; H^{1,1}_{\BC}(Z)/H^{1,1}_{\BC}(Y).
\]
Therefore if $K_Z+B$ is $f$-pseudo-effective, then there exists a K\"ahler class $\alpha$ such that for any $\epsilon>0$ there exists  an $f$-big $\mathbb R$-line bundle $L_{\epsilon}$ and class $\gamma_{\epsilon}\in H^{1,1}_{\BC}(Y)$ such that  
\begin{equation}\label{eq--consequence of relative pseudo-eff}
     K_{Z}+B-f^*(K_Y)+\epsilon \alpha=L_{\epsilon}+f^*(\gamma_{\epsilon}).
\end{equation}Perturbing the coefficients of $L_{\epsilon}$, we may assume it is a $\mathbb Q$-line bundle. Taking log resolutions and we can get the following commutative diagram of projective fibrations
   \begin{equation}
         \begin{tikzcd}
Z \arrow[d, "f"'] & Z' \arrow[l, "\nu"'] \arrow[d, "f'"] \\
Y & Y' \arrow[l, "\mu"'].
\end{tikzcd}
    \end{equation}Then there exist $\nu$-exceptional effective $\mathbb Q$-divisors $E_1$ and $E_2$ with simple normal crossing support and $\mu$-exceptional $\mathbb Q$-divisor $F$ such that 
    \begin{equation}
    \begin{aligned}
             K_{Z'}+(\nu^{-1})_*B+E_1&=\nu^*(K_Z+B)+E_2,\\
             K_{Y'}&=\mu^*(K_Y)+F.
    \end{aligned}
    \end{equation}Since $(Z,B)$ is klt and $Y$ is canonical, we know that $\lfloor (\nu^{-1})_*B+E_1\rfloor=0$ and $F\geq 0$. Then by \eqref{eq--consequence of relative pseudo-eff}, we can write 
    \begin{equation}
        K_{Z'}+(\nu^{-1})_*B-(f')^*(K_{Y'})+E_1+\epsilon\nu^*(\alpha)=\nu^*(L_{\epsilon})+E_2+(f')^*(\mu^*(\gamma_{\epsilon})-F).
    \end{equation}
Then using that $\nu^*(\alpha)$ is big and nef, by taking further resolution, we can assume $$\nu^*(\alpha)=\omega'+F',$$ where $\omega'$ is a K\"ahler class on $Z'$ and $F'$ is an effective $\mathbb Q$-divisor such that $(\nu^{-1})_*B+E_1+F'$ has simple normal crossing support. Then we can take $\epsilon\in \mathbb Q_+$ sufficiently small 
such that 
\begin{equation}
    \lfloor (\nu^{-1})_*B+E_1+\epsilon F'\rfloor=0.
\end{equation} Perturbing the coefficients of $(\nu^{-1})_*B$ and $E_1$  to rational numbers and using the K\"ahler form $\omega'$ to absorb the error term, we obtain that there exists an effective $\mathbb Q$-divisor $D$ with simple normal crossing support and $\lfloor D\rfloor =0$ and an a K\"ahler class $\omega''$ on $Z'$ such that 
\begin{equation}
    K_{Z'}+(\nu^{-1})_*B-(f')^*(K_{Y'})+E_1+\epsilon\nu^*(\alpha)= K_{Z'}-(f')^*(K_{Y'})+D+\omega''
\end{equation}
Similarly, we can approximate \( E_2 \) by effective \( \mathbb{Q} \)-divisors and use \( \omega'' \) to absorb the error term.  
Since pseudo-effectivity is a closed property, we may therefore assume in the following that \( E_2 \) is an effective \( \mathbb{Q} \)-divisor.

Let $L=\nu^*(L_{\epsilon})+E_2$. Since $\nu^*(L_{\epsilon})$ is a relative big $\mathbb Q$-line bundle, by Lemma \ref{lem-relative big line bundle has fiberwise sections}, we know that for $m$ sufficiently divisible, $mL$ has non-empty fiberwise sections over a general point $y'\in Y'$. 
Then, by applying \cite[Theorem~6.2]{hacon-paun}, we obtain that
\begin{equation}
    K_{Z'} + (\nu^{-1})_* B - (f')^*(K_{Y'}) + E_1 + \epsilon \, \nu^*(\alpha)
\end{equation}
is a pseudoeffective class on \( Z' \). Note that, we have 
\begin{equation}
     K_{Z'} + (\nu^{-1})_* B - (f')^*(K_{Y'}) + E_1 + \epsilon \, \nu^*(\alpha)=\nu^*(K_Z+B-f^*(K_Y)+\epsilon \alpha)+E_2-(f')^*F.
\end{equation}
Since $E_2$ is $\nu$-exceptional and  $(f')^*F$ is effective, we get 
\begin{equation}
    K_Z+B-f^*(K_Y)+\epsilon \alpha
\end{equation}is a pseudo-effective class on $Z$. Let $\epsilon\rightarrow 0$, we obtain that $K_{Z}+B-f^*(K_Y)$ is pseudo-effective. A contradiction.

\end{proof}

\noindent\textit{Proof of Theorem \ref{thm-weak basepointfree}.}
Since we have assumed $K_X$ is not pseudo-effective, then by the result in \cite{ou2025}, $X$ is uniruled and hence one can consider its MRC fibration \cite{campana2004,KMM,campana92}. By the elimination of points of indeterminacy \cite[Theorem 2.1.24]{ma2007} and the results in \cite{claudon-horing}, we know that there exist projective fibrations
\begin{equation}
X\xlongleftarrow{\mu_0}X_0\xlongrightarrow{f_0} Y_0 
\end{equation} such that $\mu$ is a composition of a finite succession of blow-ups with smooth centered and $\pi_0$ is a model of the MRC fibration of $X$. In particular, the general fibers of $f_0$ are rationally connected. 

Although the following argument also applies when \( Y_0 \) is a point, in that case there is a more direct argument.  
Indeed, when \( Y_0 \) is a point, the compact K\"ahler manifold \( X \) is rationally connected, hence projective, and satisfies \( H^{2,0}(X) = 0 \).  
Therefore, \( \alpha \) is represented by a big and nef \( \mathbb{R} \)-divisor, and the desired property then follows from the existence of log terminal models for klt pairs with big boundary \cite[Theorem 1.2]{BCHM} and the base-point-freeness theorem for \( \mathbb{R} \)-divisors; see, for example, \cite[Theorem~3.9.1]{BCHM}; see also \cite[Exercise 5.10]{hacon-book}.

Since \( \alpha \) is big and nef, after replacing \( X \) by a suitable sequence of blow-ups, we may assume that
\[
\alpha_0:=\mu_0^*(\alpha) = B_0 + \omega_0,
\]
where \( B_0 \) is an effective \( \mathbb{Q} \)-divisor with simple normal crossing support and \( \lfloor B_0 \rfloor = 0 \), and \( \omega_0 \) is a Kähler class.  
In particular, the pair \( (X_0, B_0) \) is klt. We gather the following facts, which will be used repeatedly.
\begin{itemize}
  \item $K_{X_0}+B_0+\omega_0=K_{X_0}+\mu_0^*(\alpha)$ is not big, since $K_X+\alpha$ is not big.
  \item $K_{X_0}+B_0$ is not $f_0$-pseudo-effective. By Lemma~\ref{lem-not relatively psef}, this follows because $K_{X_0}+B_0$ is not pseudo-effective while $K_{Y_0}$ is pseudo-effective \cite[Theorem~1.1, Lemma~8.10]{ou2025}.
  \item $K_{X_0}+B_0+\omega_0$ is pseudo-effective, since it is the sum of the pullback of the pseudo-effective class $K_X+\alpha$ and an effective exceptional divisor.
  \item $\omega_0$ defines a relatively ample  $\mathbb R$-line bundle in $N^1(X_0/Y_0)$ since general fibers of $f_0$ are rationally connected; see the discussion in Section \ref{sec:preliminary}.
  \item $(X_0, B_0+\omega_0)$ a generalized klt pair since $\alpha_0$ is nef \cite[Remark 2.10]{dhy}.
\end{itemize}

We first run a relative \( (K_{X_0} + B_0) \)-MMP with scaling of \( \omega_0 \); see Theorem~\ref{thm-relative mmp}.  
Note that since \( X_0 \) is strongly \( \mathbb{Q} \)-factorial \cite[Lemma~2.3]{dh-3fold}, all the spaces appearing in the process remain strongly \( \mathbb{Q} \)-factorial by Lemma~\ref{q-factroial is preserved}.
This yields the following sequence
\begin{equation}
    X_0 \xdashrightarrow{p_0} X_{0,1} \xrightarrow[]{q_1} X_1,
\end{equation}
where \( p_0 \) is a composition of divisorial contractions and flips, and \( q_1 \) is a Mori–Fano fibration.  
On \( X_{0,1} \), we define \( B_{0,1} := (p_0)_*(B_0) \) and \( \alpha_{0,1} := (p_0)_*(\omega_0) \).  
There exists a constant \( \lambda_{0,1} > 0 \) such that \( K_{X_{0,1}} + B_{0,1} + \lambda_{0,1}\alpha_{0,1} \) is nef over \( Y_0 \) and \( q_1 \)-trivial.  
By Lemma~\ref{lem--bigness preserved under pushforward}, the class \( \alpha_{0,1} \) is big.  
We claim that
\begin{equation}
    \lambda_{0,1} \le 1.
\end{equation}
Indeed, by applying Lemma~\ref{lem--more effective before contraction} to each divisorial contraction and flip,  
we obtain a common resolution \( X' \) of \( X_0 \) and \( X_{0,1} \) such that for any \( t \in [0, \lambda_{0,1}] \),  
there exists an effective \( \nu_{0,1} \)-exceptional \( \mathbb{R} \)-divisor \( E_t \) satisfying
\begin{equation}\label{eq--more positivity before contraction}
   \nu_0^*(K_{X_0} + B_0 + t\omega_0)
   = \nu_{0,1}^*(K_{X_{0,1}} + B_{0,1} + t\alpha_{0,1}) + E_t.
\end{equation}
If \( \lambda_{0,1} \geq  1 \), then since \( E_1 \) is \( \nu_{0,1} \)-exceptional, we deduce that  
\( K_{X_{0,1}} + B_{0,1} + \alpha_{0,1} \) is pseudoeffective.  Then since $\alpha_{0,1}$ is big, we would have  \( K_{X_{0,1}} + B_{0,1} + \lambda_{0,1}\alpha_{0,1} \) is big if $\lambda_{0,1}>1$.
This leads to a contradiction, since  $K_{X_{0,1}} + B_{0,1} + \lambda_{0,1}\alpha_{0,1}$
is a pullback from a lower-dimensional Kähler space and hence not big.

By Lemma~\ref{lem-generalized klt is preserved}, the generalized pair \( (X_{0,1}, B_{0,1} + \lambda_{0,1} \alpha_{0,1}) \) remains generalized klt. 
Then, by the discussion in Section~\ref{sec-definiton of pushforward}, relying on \cite[Theorem~0.3]{hacon-paun}, there exists a strongly $\mathbb Q$-factorial generalized klt pair \( (X_1, B_1 + \alpha_1) \) such that 
\begin{equation}\label{change of class under first mori-fano fibration}
    K_{X_{0,1}} + B_{0,1} + \lambda_{0,1} \alpha_{0,1}
    = q_1^*(K_{X_1} + B_1 + \alpha_1).
\end{equation}
By Lemma \ref{lem-bigness if preserved under fano fibration}, we know that $B_1+\alpha_1$ is a  big class.
If $K_{X_1}+B_1+\alpha_1$ is big on $X_1$, then we stop. Otherwise, we know that $K_{X_1}$ is not pseudo-effective and we use Lemma \ref{lem-not relatively psef} to conclude that $K_{X_1}$ is not pseudo-effective over $Y_0$. 

Since $K_{X_1}+B_1+\alpha_1$ is relatively  nef over $Y_0$, we are able to continue running the relative MMP for $K_{X_1}$ with scaling of the big class $B_{1}+\alpha_1$ to get a Mori-Fano fibration after finite divisorial contractions and flips. Repeat this process, we get the following.
 For $k=1,\cdots, l$, there exists strongly $\mathbb Q$-factorial klt pairs  $(X_{k},B_{k})$ and $(X_{k-1,k},B_{k-1,k})$ pojective over $Y_0$ and bimeromorphic map $p_{k-1}:X_{k-1}\dashrightarrow X_{k-1,k}$, which is a composition of divisorial contractions and flips and $q_k:X_{k-1,k}\rightarrow X_{k}$ is a Mori-Fano fibration.  They satisfy the following property
  \begin{itemize}
      \item there exist positive real numbers 
      \begin{equation}
        1=\lambda_1 \geq \lambda_{2}\geq \cdots \geq \lambda_{l}      \end{equation}and define $B_{k-1,k}=(p_{k-1})_*(B_{k-1})$,  $\alpha_{k-1,k}=(p_{k-1})_*(\alpha_{k-1})$ such that 
        \begin{equation} \label{eq-change of class under mori-fano fibration} 
        K_{X_{k-1,k}}+\lambda_k(B_{k-1,k}+\alpha_{k-1,k})=q_k^*(K_{X_k}+\lambda_k(B_k+\alpha_k))\end{equation} is nef over $Y_0$ and it is $q_k$-trivial
  
      \item $B_i+\alpha_i$ is a big class for $i=1,\cdots, l$ by Lemma \ref{lem--bigness preserved under pushforward} and Lemma \ref{lem-bigness if preserved under fano fibration},  

      \item $K_{X_{i}}+\lambda_i(B_i+\alpha_i)$ is not big for $i=0,\cdots l-1$.
  \end{itemize}

Since each Mori–Fano fibration strictly decreases the dimension by one, we know that after finitely many steps, we arrive at one of the following situations:
\begin{enumerate}
    \item \( K_{X_l} + \lambda_l (B_l + \alpha_l) \) is big;
    \item \( \dim X_l = \dim Y_0 \).
\end{enumerate}
If case~(2) occurs, then using the fact that $Y_0$ is smooth and \( K_{Y_0} \) is pseudoeffective, we can conclude that \( K_{X_l} \) is also pseudoeffective, and hence \( K_{X_l} + \lambda_l (B_l + \alpha_l) \) is big.  
In summary, we always obtain that there exists a positive constant \( \lambda_l \le 1 \) such that
\begin{equation}
    K_{X_l} + \lambda_l (B_l + \alpha_l) \text{ is big.}
\end{equation}

Then we claim that 
\begin{equation}
    \lambda_l=\cdots=\lambda_2=\lambda_{0,1}=1.
\end{equation}
We firstly show $\lambda_l=\cdots=\lambda_2$. Since otherwise we can pick the largest $k<l$ such that that $\lambda_k>\lambda_l$. Then using Lemma \ref{lem--more effective before contraction} and tracing the change of the class $K_X+B+\alpha$ in each divisorial contrtaction, flip and Mori-Fano fibration, we get that $K_{X_i}+\lambda_l(B_i+\alpha_i)$ is pseudo-effective for all $i\geq k+1$ and moreover similar to \eqref{eq--more positivity before contraction}, we get that there exists a common resolution $X'$ of $X_k$ and $X_{k,k+1}$
such that 
\begin{equation}
    \nu_k^*(K_{X_k}+\lambda_l(B_k+\alpha_k))=\nu_{k,k+1}^*(K_{X_{k,k+1}}+\lambda_l(B_{k,k+1}+\alpha_{k,k+1}))+F
\end{equation}for some $\nu_{k,k+1}$-exceptional effective $\mathbb R$-divisor $F$. Then using \eqref{eq-change of class under mori-fano fibration} and the choice of $k$, we know that $K_{X_k}+\lambda_l(B_k+\alpha_k)$ is pseudo-effective. Since $B_k+\alpha_k$ is big and $\lambda_k>\lambda_l$, we would obtain that $ K_{X_k}+\lambda_k(B_k+\alpha_k)$ is indeed big. This contradicts with our choice of $l$. As a consequence of $\lambda_l=\cdots=\lambda_2$, similar argument as above shows that this number has to be 1 and we know that $K_{X_1}+B_1+\alpha_1$ is pseudo-effective. Then we can show that $\lambda_{0,1}=1$. If $\lambda_{0,1}<1$, then by \ref{eq--more positivity before contraction} and \ref{change of class under first mori-fano fibration}, this would imply $K_{X_0}+B_0+\omega_0$ is big, which contradicts with our initial assumption \eqref{eq--initial assumption}. 

Taking resolutions of $X_l$ and using Demailly's regularization theorem for the big class as mentioned in Section \ref{sec:preliminary}--(10), we know that after a further resolution, we get a bimeromorphic morphism $\pi$ such that there exists a K\"ahler class $\omega_Y$ and an effective $\mathbb Q$-divisor on $Y$ such that 
\begin{equation}
    \pi^*(K_{X_l}+B_l+\alpha_l)=\omega_Y+D_Y.
\end{equation}Then taking a common resolution of each $X_k$ and $X_{k,k+1}$ and using elimination of indeterminacy, we can get the following diagram
\begin{equation}
    \begin{tikzcd}[column sep=0.3em, row sep=normal]
&[1em]&   Z:=Z_0  \arrow[dl,"\nu_0"'] \arrow[dr,"\nu_{0,1}"]\arrow[rrr] &  &[1em]&Z_1  \arrow[dl] &\cdots &[1em]&   Z_{l-1}  \arrow[dl] \arrow[dr] \arrow[rr] &  &[1em]Z_l=:Y\arrow[d,"\pi"]\\
X &\arrow[l, "\mu_0"']  X_0 \arrow[rr, dashed] &                         & X_{0,1}\arrow[r]  & X_1& \cdots &\cdots & X_{l-1} \arrow[rr, dashed] &                         & X_{l-1,l}\arrow[r]  & X_l
\end{tikzcd}
\end{equation}

Let  $f:Z\rightarrow Y$ be the induced morphism from the composition $Z_{k-1}\rightarrow Z_k$ in the above commutative diagram and let $\mu:Z\rightarrow X$ be the composition of $\nu_0:Z\rightarrow X_0$ and $\mu_0$. Note that by the construction, each map $Z_{k-1}\rightarrow Z_k$ is bimeromorphic to a projective fibration, therefore itself is a Moishezon fibration i.e., there exists a relative big line bundle. Then by \cite[Theorem 1.1]{claudon-horing}, we know that it is a projective fibration and hence $f:Z\rightarrow Y$ is a projective fibration. All these morphisms are taken over \( Y_0 \), and since \( f_0 : X_0 \rightarrow Y_0 \) has rationally connected general fibers, the general fiber of \( f \) is also rationally connected.

By the construction, we obtain
\begin{equation}\label{eq--class realtion on nef model}
    \nu_{0,1}^*(K_{X_{0,1}} +B_{0,1}+ \alpha_{0,1}) = f^*(\omega_Y + D_Y).
\end{equation}
Moreover, there exist a \( \mu \)-exceptional \( \mathbb{R} \)-divisor \( E_1 \) and a \( \nu_{0,1} \)-exceptional \( \mathbb{R} \)-divisor \( E_2 \), having no common components, such that 
\begin{equation}\label{eq--class relation for X and x01}
    \mu^*(K_X + \alpha) + E_1
    = \nu_{0,1}^*(K_{X_{0,1}} +B_{0,1}+ \alpha_{0,1}) + E_2.
\end{equation}
Combining \eqref{eq--class relation for X and x01} with \eqref{eq--class realtion on nef model}, we obtain
\begin{equation}
    \mu^*(K_X + \alpha)
    = f^*(\omega_Y + D_Y) + E_2 - E_1.
\end{equation}
Applying the negativity lemma to $\mu:Z\rightarrow X$, we obtain that $f^*(D_Y)+E_2-E_1\geq 0$. Then we obtain the desired effective $\mathbb 
R$-divisor 
\begin{equation}
    D:=f^*(D_Y)+E_2-E_1.
\end{equation}
Moreover if $K_{X}+\alpha$ is nef, then applying the negativity lemma to $\nu_{0,1}:Z\rightarrow X_{0,1}$, we get that in \eqref{eq--class relation for X and x01}, we have $E_2=0$. Therefore in this case there exists a $\mu$-exceptional effective $\mathbb R$-divisor $E_1$ such that
\begin{equation}
    E_1+D=f^*(D_Y)
\end{equation}
\qed

\subsection{A general version and applications}
For later reference, we record the following more general version of Theorem~\ref{thm-weak basepointfree} and include a brief sketch of the proof, highlighting only the differences from the argument given above.


\begin{theorem}\label{thm--generalized klt version}
    Let $(X, B+\alpha)$ be a compact K\"ahler strongly $\mathbb Q$-facotial generalized klt pair such that $B+\alpha$ is big. If $K_X+B+\alpha$ is pseudo-effective, then there exists projective fibration $X\xlongleftarrow{\mu} Z\xlongrightarrow{f}Y$ with $Z,\, Y$ being compact K\"ahler manifolds and general fibers of $f$ being rationally connected such that
         \begin{equation}
             \mu^*(K_X+B+\alpha)=f^*(\alpha_1)+D,
         \end{equation} where $ \alpha_1$ is a K\"ahler class and $D$ is an effective $\mathbb R$-divisor.
      Moreover if $K+B+\alpha$ is nef, then there exists an effective $\mathbb Q$-divisor $D_Y$ on $Y$ and an effective $\mu$-exceptional $\mathbb R$-divisor $E$ such that 
      \begin{equation}
          E+D=f^*(D_Y).
      \end{equation}
\end{theorem}

\begin{proof}
     It is clear from the proof above that the bigness and nefness of \( \alpha \) are used only at the initial stage, to ensure that after taking a resolution we obtain a generalized klt pair whose boundary class is Kähler, allowing us to start the relative MMP process.  
Therefore, by a standard argument (see for example \cite[Lemmas~3.1 and~3.4]{das-hacon2024}), this condition on \( \alpha \) can be replaced by the weaker assumption that \( B+\alpha \) is big and \( (X, B+\alpha) \) is a generalized klt pair \cite{dhy, hacon-paun}.  

For the reader’s convenience, we recall some details here.  
By the definition of a generalized klt pair, there exists a bimeromorphic morphism
\(
\nu : X' \rightarrow X
\)
from a Kähler manifold \( X' \), and an \( \mathbb{R} \)-divisor \( B' \) with simple normal crossing support such that \( B = \nu_* B' \ge 0 \), together with a nef class \( \alpha' \) on \( X' \) satisfying
\begin{equation}
    K_{X'} + B' + \alpha' = \nu^*(K_X + B + \alpha),
    \qquad \lfloor B' \rfloor \leq  0.
\end{equation}
     
     Then under this assumption, we can write $B'=B'_+-B'_-$ and there is an effective $\nu$-exceptional $\mathbb R$-divisor $E$ such that 
    \begin{equation}
        \alpha'=\nu^*(B+\alpha)-(\nu^{-1})_*(B)-E.
    \end{equation}Since $B+\alpha$ is big, then passing to a further blow-ups, we know that there exists effective $\mathbb Q$-divisor $F$, such that $B'_++B'_-+E+F$ has simple normal crossing support and a K\"ahler class $\omega'$ on $X'$ such that 
    \begin{equation}
        \nu^*(B+\alpha)=\omega'+F.
    \end{equation}Then we can write 
    \begin{equation*}
        \alpha'=(1-\epsilon)\alpha'+\epsilon(\omega'+F-(\nu^{-1})_*(B)-E)
    \end{equation*}and hence 
    \begin{equation}\label{eq--relatiion on resolution for generalized}
     K_{X'}+B'_+-\epsilon(\nu^{-1})_*(B)+\epsilon F+(1-\epsilon)\alpha'+\epsilon\omega'=\nu^*(K_X+B+\alpha)+B_-+\epsilon E.
    \end{equation}Note that by choosing \( \epsilon > 0 \) sufficiently small, we have  
\begin{equation}\label{eq--klt condition}
    B'_+ - \epsilon (\nu^{-1})_*(B) \ge 0, 
    \qquad 
    \lfloor B'_+ - \epsilon (\nu^{-1})_*(B) + \epsilon F \rfloor = 0.
\end{equation}
Moreover, by perturbing the coefficients slightly to rational numbers and using the Kähler class 
\( (1 - \epsilon)\alpha' + \epsilon \omega' \) to absorb the resulting error term, 
we may assume that 
\[
    G' := B'_+ - \epsilon (\nu^{-1})_*(B) + \epsilon F
\]
is an effective \( \mathbb{Q} \)-divisor in the following.
    Therefore by \eqref{eq--relatiion on resolution for generalized} and \eqref{eq--klt condition}, we get a K\"ahler class $\gamma'=(1-\epsilon)\alpha'+\epsilon\omega' $ and  a klt pair $(X', G')$ such that 
    \begin{equation}
        K_{X'}+G'+\gamma' \text{ is pseudo-effective.}
    \end{equation}

    Note that $B_-'+\epsilon E$ are $\nu'$-exceptional, then indeed one can follow the same argument given in Theorem \ref{thm-weak basepointfree}. That is we first run a relative $(K_{X'}+G')$-MMP with scaling of the K\"ahler class $\gamma'$. Then we get a constant $\lambda_{0,1}\leq 1$ and a Mori-Fano fibration $X_{0,1}\rightarrow X_1$ to a lower dimensional generalized klt pair $  (X_1, G_1+\gamma_1)$ with $G_1+\gamma_1$ big and $K_{X_1}+G_1+\gamma_1$ relatively nef. Then we run relative $K_{X_1}$-MMP with scaling of $G_1+\gamma_1$. Then the same argument as before ensures that $\lambda_{0,1}=1$ and we get projective fibrations  $f:Z\rightarrow Y$, $\nu_{0,1}:Z\rightarrow X_{0,1}$ and $\mu:Z\rightarrow X$ such that 
\begin{equation}\label{eq-relation on x01-2}
\nu_{0,1}^*(K_{X_{0,1}}+B_{0,1}+\gamma_{0,1})=f^*(\omega_Y+D_Y).
\end{equation}
Moreover since the $\mathbb R$-divisor \(B'_- + \epsilon E \) in \eqref{eq--relatiion on resolution for generalized} are exceptional, we know that there exist a \( \mu \)-exceptional \( \mathbb{R} \)-divisor \( E_1 \) and a \( \nu_{0,1} \)-exceptional \( \mathbb{R} \)-divisor \( E_2 \), with no common components, such that 
\begin{equation}\label{eq--class relation for X and x01 2}
    \mu^*(K_X + B + \alpha) + E_1
    = \nu_{0,1}^*(K_{X_{0,1}} + B_{0,1} + \gamma_{0,1}) + E_2.
\end{equation}
Combining \eqref{eq--class relation for X and x01 2} with \eqref{eq-relation on x01-2}, we obtain
\begin{equation}
    \mu^*(K_X + B + \alpha)
    = f^*(\omega_Y + D_Y) + E_2 - E_1.
\end{equation}
Applying the negativity lemma to $\mu:Z\rightarrow X$, we obtain that $f^*(D_Y)+E_2-E_1\geq 0$. Then we obtain the desired effective $\mathbb 
R$-divisor 
\begin{equation}
    D:=f^*(D_Y)+E_2-E_1.
\end{equation}
Moreover if $K_{X}+B+\alpha$ is nef, then applying the negativity lemma to $\nu_{0,1}:Z\rightarrow X_{0,1}$, we get that in \eqref{eq--class relation for X and x01 2}, we have $E_2=0$. 
\end{proof}

   Although Theorem~\ref{thm-weak basepointfree} and Theorem \ref{thm--generalized klt version} are only weaker versions of the transcendental base-point-freeness result, it already provides useful information about the adjoint class \( K_X + \alpha \).  
In particular, we give the following remarks.

\begin{remark}Under the same assumptions as in Theorem~\ref{thm--generalized klt version}, we collect the following basic properties:
\begin{itemize}
    \item The pseudoeffective class \( K_X +B+ \alpha \) admits a positive \((1,1)\)-current whose singularities are contained in an analytic subset.
    
    \item When \( K_X + B+\alpha \) is big, by \cite[Theorem~1.4]{boucksom2002volume}, taking \( f \) to be the identity, for any \( \epsilon > 0 \) there exists a modification 
    \( \mu_{\epsilon} : X_{\epsilon} \to X \) such that 
    \[
        \mu_{\epsilon}^*(K_X+B+\alpha) = \alpha_{\epsilon} + D_{\epsilon},
    \]
    where \( \alpha_{\epsilon} \) is a Kähler class and \( D_{\epsilon} \) is an effective \( \mathbb{Q} \)-divisor satisfying 
    \begin{equation}
        \operatorname{vol}(K_X + B+\alpha) - \epsilon 
        \leq \operatorname{vol}(\alpha_{\epsilon}) 
        \leq \operatorname{vol}(K_X + B+\alpha).
    \end{equation}
    
    \item When \( K_X +B+ \alpha \) is not big, the superadditivity of the positive intersection product 
    \cite[Theorem~3.5]{bdpp} implies that the class $\alpha_1$ and the effective $\mathbb R$-divisor obtained in Theorem \ref{thm--generalized klt version} satisfying the following
    \begin{equation}
        \big\langle f^*(\alpha_1)^k \cdot D^{n-k} \big\rangle = 0
        \quad \text{for all } 0 \le k \le n,
    \end{equation}
    where \( \langle \cdot \rangle \) denotes the positive intersection product defined in \cite{bdpp}.
\end{itemize}
\end{remark}

 Moreover, we can show that in Theorem~\ref{thm--generalized klt version},  
the dimension of the base coincides with the numerical dimension of the pseudoeffective class \( K_X +B+ \alpha \) as defined in \cite{bdpp}; see also  \cite{BEGZ}.  
Recall that for a pseudoeffective class \( \beta \in H^{1,1}(X, \mathbb{R}) \) on a compact Kähler manifold,
\[
\operatorname{nd}(\beta)
:= \max \left\{
    k \in \mathbb{N}
    \,\middle|\,
    \langle \beta^k \rangle \neq 0
    \text{ in } H^{k,k}(X, \mathbb{R})
\right\}.
\]
For a pseudoeffective class \( \beta \) on a singular compact normal Kähler space, we define its numerical dimension as the numerical dimension of its pullback to a Kähler resolution of that space.
In what follows, we will only need some basic properties of the numerical dimension:
\begin{itemize}
    \item for a nef class $\beta$, the positive intersection product is the same as the usual cup product;
    \item if \( \gamma \) is a pseudoeffective class on \( X \), then  
\begin{equation}
    \nd(\beta + \gamma) \ge \nd(\beta).
\end{equation}
\item Let \( \phi : X \to V \) be a bimeromorphic morphism from a compact Kähler manifold \( X \) onto a compact normal Kähler space \( V \).  
Suppose there exists a \( \phi \)-exceptional effective \( \mathbb{R} \)-divisor \( F \) such that  
\( \phi^*(\omega_V) - F \) is a Kähler class on \( X \) for some Kähler class \( \omega_V \) on \( V \).  
Then, for any class \( \alpha \in H^{1,1}_{\BC}(V) \) and any effective \( \phi \)-exceptional \( \mathbb{R} \)-divisor \( E \), we have
\begin{equation}\label{eq-numerical class does not change adding exceptional}
    \nd\bigl(\phi^*(\alpha) + E\bigr) = \nd\bigl(\phi^*(\alpha)\bigr).
\end{equation}
Indeed, given the assumption and by the definition of the positive intersection product for pseudo-effective classes, it suffices to show that for any $\epsilon>0$ and $k\in \mathbb N$
\begin{equation}\label{eq-positive intersection}
    \langle (\phi^*(\alpha + \epsilon \omega_V) + E)^k \rangle
    = \langle \phi^*(\alpha + \epsilon \omega_V)^k \rangle.
\end{equation}
By the argument given in Lemma~\ref{lem--adding exceptional no effect}, we know that if  
\( \phi^*(\alpha + \epsilon \omega_V) + E \) admits a positive \((1,1)\)-current \( T \) such that \( T \ge \beta \) for some smooth semi-positive \((1,1)\)-form \( \beta \),  
then \( \phi^*(\alpha + \epsilon \omega_V) \) also admits a positive \((1,1)\)-current \( \phi^*(T_V) \) satisfying \( \phi^*(T_V) \ge \beta \).  
Hence, \eqref{eq-positive intersection} follows from this property.
\end{itemize}

\begin{proposition}\label{prop-dim of base coincide with numerical}
  Under the same assumptions as in Theorem~\ref{thm--generalized klt version}, let  $X \xlongleftarrow{\mu} Z \xlongrightarrow{f} Y$ 
be the projective fibration obtained there.  
Then we have  
\begin{equation}
    \dim Y = \nd(K_X +B+ \alpha).
\end{equation}
\end{proposition}

\begin{proof}
Recall \eqref{eq--class relation for X and x01 2} and \eqref{eq-relation on x01-2}, in the proof of Theorem~\ref{thm--generalized klt version}, and  we have the following identity of cohomology classes:
\begin{equation}
    \mu^*(K_X +B+ \alpha) = f^*(\alpha_1 + D_Y) + E_2 - E_1,
\end{equation}
where \( \alpha_1 \) is a Kähler class on \( Y \), and \( E_1, E_2, D_Y \) are effective \( \mathbb{R} \)-divisors satisfying \[ f^*(D_Y) + E_2 - E_1 \ge 0. \]  
Moreover, $E_1$ is $\mu$-exceptional and \( E_2 \) is \( \nu_{0,1} \colon Z \to X_{0,1} \)-exceptional.  
By \cite[§8]{kollár2021}, after passing to a further resolution, we may assume that both 
\( \nu_{0,1} \) and \( \mu \) admit effective exceptional \( \mathbb{R} \)-divisors \( F_{\nu} \) and \( F_{\mu} \), respectively, such that 
\( -F_{\nu} \) and \( -F_{\mu} \) are relatively ample.  
Hence, we can apply the property \eqref{eq-numerical class does not change adding exceptional} in the following computation:
\begin{equation}
\begin{aligned}
    \nd(K_X +B+ \alpha)
    &= \nd\bigl(\mu^*(K_X +B+ \alpha) + E_1\bigr) \\
    &= \nd\bigl(\nu_{0,1}^*(K_{X_{0,1}} + B_{0,1}+\alpha_{0,1}) + E_2\bigr) \\
    &= \nd\bigl(\nu_{0,1}^*(K_{X_{0,1}} + B_{0,1}+\alpha_{0,1})\bigr) \\
    &= \nd\bigl(f^*(\alpha_1 + D_Y)\bigr)=\dim Y.
\end{aligned}
\end{equation}
\end{proof}

For later reference, we state the folklore base-point-freeness conjecture for generalized pairs on compact strongly \( \mathbb{Q} \)-factorial Kähler spaces.
\begin{conjecture}\label{conj--basepointfree}
    Let \( (X, B + \alpha) \) be a compact Kähler strongly \( \mathbb{Q} \)-factorial generalized klt pair.  
If \( B + \alpha \) is big and \( K_X + B + \alpha \) is nef, then \( K_X + B + \alpha \) is semiample; that is, there exists a holomorphic morphism \( f : X \to Y \) onto a compact Kähler space and a Kähler class \( \omega_Y \) on \( Y \) such that
\[
K_X + B + \alpha = f^*(\omega_Y) \text{ in } H^{1,1}_{\BC}(X).
\]
\end{conjecture}
\begin{remark}\label{rmk-enough to show freeness for big class}
Theorem~\ref{thm--generalized klt version} suggests that if Conjecture~\ref{conj--basepointfree} holds under the additional assumption that \( K_X + B + \alpha \) is big, then it should hold in full generality.  
Indeed combining with the results of~\cite{ou2025}, the cone theorem for Kähler \( \mathbb{Q} \)-factorial generalized klt pairs has been established in~\cite[Theorem~0.5]{hacon-paun}.  
Therefore, assuming Conjecture~\ref{conj--basepointfree} in the big adjoint class case suffices to construct contractions of negative extremal rays, allowing one to run a \emph{global} MMP with scaling.  
Since the boundary class is big, one may then generalize the arguments of~\cite{fujino2022, Das_HP2024, BCHM} to show that flips exist and terminate, thereby obtaining a good log terminal model.  
Applying this procedure to the pair \( (Y, \omega_Y + D_Y) \) obtained in Theorem~\ref{thm--generalized klt version}, and replacing \( Y \) by this good minimal model, we obtain the following diagram
\begin{equation}
    \begin{tikzcd}
        & Z \arrow[dl, "\nu"'] \arrow[dr, "h"] \\
        X \arrow[rr, dashed, "f"'] && Y.
    \end{tikzcd}
\end{equation}
Moreover, there exist a Kähler class \( \omega_Y \), an effective \( \nu \)-exceptional \( \mathbb{R} \)-divisor \( E \), and an effective \( \mathbb{R} \)-divisor \( F \) with \( \operatorname{codim}_Y(h(F)) \ge 2 \) such that  
\begin{equation}
    \nu^*(K_X + B + \alpha) + E = h^*(\omega_Y) + F.
\end{equation}
We may assume that \( E \) and \( F \) have no common components.  
Applying the negativity lemma~\cite[Lemma~1.3]{wang2021iitaka} together with~\cite[Lemma~2.8]{das-hacon2024}, we conclude that \( E = 0 \) and \( F = 0 \).
Since \( \omega_Y \) is a Kähler class, it follows that \( h \) must contract every fiber of \( \nu \).  
By~\cite[Lemma~1.15]{debarre}, the morphism \( h \) therefore factors through \( \nu \); that is, the map \( f \) extends to a globally defined holomorphic morphism \( f : X \to Y \).  
Finally, since the pullback map \( \nu^* : H^{1,1}_{\BC}(X) \to H^{1,1}_{\BC}(Z) \) is injective, we conclude that
\[
    K_X + B + \alpha = f^*(\omega_Y).
\]
\end{remark}

 Since the existence of log terminal and log canonical models for generalized klt pairs with big adjoint classes 
on two- and three-dimensional compact Kähler spaces has been established in 
\cite[Theorems~1.2 and~2.33]{dhy} (see also~\cite{horing-peternell,CHP16,das-Ou,dh-3fold}), 
it follows from the discussion in Remark~\ref{rmk-enough to show freeness for big class} 
and Proposition~\ref{prop-dim of base coincide with numerical} that 
Conjecture~\ref{conj--basepointfree} holds whenever the numerical dimension of \( K_X + B + \alpha \) is at most three.
In particular, we obtain the following consequence of Theorem~\ref{thm--generalized klt version}, 
which includes Corollary~\ref{cor-bpf for ndleq 3} as a special case.

\begin{corollary}\label{cor-numerical dim at most 3 is done}
Let \( (X, B + \alpha) \) be a compact Kähler strongly \( \mathbb{Q} \)-factorial generalized klt pair such that \( B + \alpha \) is big and \( K_X + B + \alpha \) is nef.  
If the numerical dimension \( \nd(K_X + B + \alpha) \le 3 \), then there exists a holomorphic morphism 
\( f : X \to Y \) onto a compact Kähler space and a Kähler class \( \omega_Y \) on \( Y \) such that
\begin{equation}
    K_X + B + \alpha = f^*(\omega_Y).
\end{equation}
\end{corollary}

Another direct consequence of Theorem~\ref{thm--generalized klt version} is the following result, which extends to the Kähler setting the analogue of \cite[Theorem~4.1]{das-hacon2024}—itself a generalization of \cite[Theorem~1.11]{birkar16}.

\begin{proposition}\label{prop-generalization of das-hacon 4.1}
    Let $(X,B+\alpha)$ be a compact K\"ahler strongly $\mathbb Q$-factorial generalized klt pair such that $B+\alpha$ is big and $K_X+B+\alpha$ is nef but not big. Then though a general point $x\in X$, there is a rational curve $\Gamma_x$ such that $(K_X+B+\alpha)\cdot \Gamma_x=0$.
\end{proposition}
\begin{proof}
 Applying Theorem~\ref{thm--generalized klt version}, we obtain projective fibrations
\[
X \xlongleftarrow{\mu} Z \xlongrightarrow{f} Y
\]
and an effective \( \mathbb{Q} \)-divisor \( D_Y \).  
Since \( K_X + B + \alpha \) is not big, we have \( \dim Y < \dim X \).  
As the general fibers of \( f \) are rationally connected, it follows that for a general point
\[
x \in Z \setminus \big( \Exc(\mu) \cup f^{-1}(\supp D_Y) \big),
\]
there exists a rational curve \( \Gamma_x \subset f^{-1}(f(x)) \) such that
\begin{equation}
    \mu^*(K_X + B + \alpha) \cdot \Gamma_x = 0.
\end{equation}
By the projection formula, this induces a rational curve \( \Gamma_x \subset X \) passing through a general point \( x \in X \) with
\[
(K_X + B + \alpha) \cdot \Gamma_x = 0.
\]
\end{proof}

\section{A generalized Schwarz lemma and proof of Theorem \ref{thm--main}}\label{sec--proof of main theorem}
In this section, we first establish a bimeromorphic version of the parabolic Schwarz lemma, generalizing the results of \cite{songtian2007, tosatti2-zhang}.  
This provides a degenerate lower bound for the evolving metric \( \omega_t \) in the Kähler–Ricci flow~\eqref{eq--ric flow}, under the assumption that the limit class \( K_X + \omega \) admits a positive \((1,1)\)-current with suitable regularity properties.  
We then apply Theorem~\ref{thm-weak basepointfree} to complete the proof of Theorem~\ref{thm--main}.

We begin by recalling some standard conventions.  
Let \( D \) be an effective Cartier divisor on a smooth complex manifold.  
Then \( D \) defines a line bundle, and we denote by \( s_D \) the tautological section of this line bundle whose zero locus is precisely \( D \).  
We fix a smooth Hermitian metric \( h_D \) on this line bundle such that \( \max |s_D|_{h_D} \le 1 \); the precise choice of \( h_D \) will not be important for our following discussion.  
When \( D = \sum a_i D_i \) is an effective \( \mathbb{R} \)-Cartier divisor, we use the notation
\[|s_D|_{h_D} := \prod_{i=1}^N |s_{D_i}|_{h_{D_i}}^{a_i}.
\]
When there is no risk of confusion, we will omit the subscript and simply write \( |s_D| \) for brevity.
\begin{theorem}\label{thm-generalized schwarz}
    Let $(X,\omega_0)$ be a compact K\"ahler manifold with $\nef_X(\omega_0)= 1$. Suppose there exists projective fibrations $X\xlongleftarrow{\mu} Z\xlongrightarrow{f} Y$ between compact K\"ahler manifolds such that
    \begin{itemize}
        \item $\mu$ is a composition of blow-ups along smooth holomorphic submanifolds such that 
        \begin{equation}\label{eq-class decomposition}
            \mu^*(K_X+\omega_0)=f^*(\alpha_1)+\alpha_2+D;
        \end{equation}
        \item there exists a K\"ahler form $\theta_1\in \alpha_1$ on $Y$ and a smooth non-negative $(1,1)$-form $\theta_2$ in $\alpha_2$. Morever $D$
     is an effective $\mathbb R$-divisor and there exists effective $\mathbb R$-divisor $D_Y$ on $Y$ and effective $\mu$-exceptional $\mathbb R$-divisor $F$ such that $D+F=f^*(D_Y)$.
    \end{itemize}Then there exists a constant $C$ such that for any $t\in [0,1)$, we have 
    \begin{equation}
         \tr_{\mu^*(\omega_t)}(f^*(\theta_1))\leq C|s_{f^*{D_Y}}|^{-C}.
    \end{equation}
\end{theorem}

\begin{remark}
Note that, in particular, this result asserts that the function \( \operatorname{tr}_{\mu^*(\omega_t)}(f^*(\theta_1)) \) is well defined and bounded on \( Z \setminus \operatorname{Supp}(f^*(D_Y)) \).  
This point requires justification, which will be provided in the proof, since \( Z \setminus \operatorname{Supp}(f^*(D_Y)) \) may still contain the exceptional locus of \( \mu \), and thus \( \mu^*(\omega_t) \) may degenerate on certain subsets of this region. 
Moreover, the degeneration of the estimate occurs only along \( \operatorname{Supp}(f^*(D_Y)) \), which in particular does not dominate \( Y \).  
This fact will play a crucial role in establishing the diameter lower bound for \( \omega_t \) later in this section.
\end{remark}

\begin{remark}
It is interesting to note that this degenerate \( C^2 \)-estimate is established using only the degenerate \( C^0 \)-estimate for the potential (Lemma~\ref{lem--C0 estimate}).  
This is in the same spirit as the estimate obtained in \cite[Section~4]{CollinsT} for the volume non-collapsing case.
\end{remark}

Note that in \eqref{eq--ric flow}, our convention is that 
\(
    [\operatorname{Ric}(\omega)] 
    = \left[ -\frac{1}{2\pi} \sqrt{-1}\,\partial \bar{\partial} \log \omega^n \right] 
    = c_1(X).
\)
Therefore, it is more natural (see~\cite{song-weinkove}) to work with the convention
\(
    \Delta_{\omega} f = \operatorname{tr}_{\omega} \!\left( \frac{\sqrt{-1}}{2\pi} \partial \bar{\partial} f \right).
\) In the following, when there is no danger of confusion, we omit the factor \( \tfrac{1}{2\pi} \).

 Fix a  smooth (1,1)-form $\alpha$ in the class $K_X+\omega_0$, then there exists a smooth volume form $\Omega$ such that 
   \begin{equation}
       \alpha-\omega_0=\frac{\ii}{2\pi}\p\pp \log \Omega,  \quad \int_X \Omega=\int_X\omega_0^n.
   \end{equation}
   We can write 
    \begin{equation}
        \omega_t=(1-t)\omega_0+t\alpha+\frac{1}{2\pi}\ii\p\pp \varphi_t,
    \end{equation}where $\varphi_t$ satisfies the  parabolic complex Monge-Amp\`ere equation:
    \begin{equation}\label{eq-parabolic equation}
        \partial_t\varphi=\log (\frac{\omega_t^n}{\Omega }), \quad \left.\varphi\right|_{t=0}=0.
    \end{equation}

   The following degenerate \( C^0 \)-estimate can be obtained by adapting the argument of \cite{zhang-zhou}; see also \cite[Section~4]{CollinsT}.  
For the reader’s convenience, we include the details below.    
    \begin{lemma}\label{lem--C0 estimate}Let \( \frac{\eta}{2\pi} \) be an \( \alpha \)-psh function.  
Then there exists a constant \( C > 0 \) (depending on \( \eta \)) independent of \( t \), such that on \( X \times [0,1) \) we have
\begin{equation}
    t\,\eta - C \leq \varphi_t \leq C.
\end{equation}
\end{lemma}
\begin{proof}
By the standard estimate on volume forms, say for example \cite[Corollary 3.2.3]{song-weinkove}, there exists a constant $C>0$ such that for any $t\in [0,1)$, we have
\begin{equation}\label{volume upper bound}
    \omega_t^n\leq C\Omega
\end{equation}
    Then the upper bound of $\varphi$ follows from \eqref{eq-parabolic equation} and \eqref{volume upper bound}.  We just need to prove the lower bound of $\varphi$ in the following.  By Demailly's regularization theorem, we know that for any $\epsilon>0$, there exists functions $\eta_{\epsilon}$ with analytic singularities with 
    \begin{equation}
        \alpha+\frac{1}{2\pi}\ii\p\pp \eta_{\epsilon}\geq-\epsilon \omega_0, \text{ and } \eta_{\epsilon}\searrow \eta \text{ as $\epsilon\searrow 0$}.
    \end{equation}For any $t'\in [0,1)$, we consider the following quantity on $X\times [0,t']$,
    \begin{equation}
        Q=(\varphi+(t'-t)\p_t\varphi+nt)-t'\eta_{\epsilon}.
    \end{equation}Note that we have
    \begin{equation}
        (\partial_t - \Delta)\varphi=\partial_t\varphi+\tr_{\omega_t}((1-t)\omega_0+t\alpha)-n
    \end{equation}and 
\begin{equation}\label{eq-parabolic equation of time derivative}
 (\partial_t - \Delta)(\partial_t \varphi)
    = \operatorname{tr}_{\omega_t}(\alpha - \omega_0).
\end{equation}
    Then we obtain that on $X\times [0,t']$, whenever $Q$ is smooth, we have for $\epsilon\in (0, 1-t']$
    \begin{equation}
        (\partial_t-\Delta)Q=\tr_{\omega}(t'\alpha+(1-t')\omega_0+\frac{t'}{2\pi}\ii\partial\pp\eta_{\epsilon})\geq 0
    \end{equation}
Since $\eta_{\epsilon}$ has a uniform upper bound independent of $\epsilon$, we can apply the maximal principle to $Q$ to obtain that there exists a constant $C$ independent of $t'$ and $\eta_{\epsilon}$ such that $Q\geq - C$ on $X\times [0,t']$. Since $\p_t\varphi$ is bounded above \eqref{volume upper bound}, we have
   $ \varphi\geq t'\eta_{\epsilon}-C$. Then we let $\epsilon$ go to 0, we obtain there is a constant $C$ indepedent of $t'$ such that on $X\times [0,t']$
   \begin{equation}\label{eq-lower bound for t'}
       \varphi\geq t'\eta-C
   \end{equation}Since $t'\in [0,1)$ is arbitrary and the constant in \eqref{eq-lower bound for t'} is independent of $t'$, therefore we obtain the desired lower bound of $\varphi$ on $X\times [0,1)$.  
\end{proof}

\begin{remark}
We remark that the lower bound for \( \varphi_t \) in this lemma is likely not optimal.  
The expected optimal lower bound is that there exists a uniform constant \( C > 0 \) such that 
\begin{equation}\label{eq--optimal lower bound for C0}
    \frac{1}{2\pi}\varphi_t \geq V_{t\alpha + (1-t)\omega_0} - C,
\end{equation}
where 
\begin{equation}
    V_{t\alpha + (1-t)\omega_0}(x)
    := \sup \left\{
        \phi(x) \,\middle|\,
        \phi \text{ is } t\alpha + (1-t)\omega_0\text{-psh},\,
        \sup_X \phi \le 0
    \right\}
\end{equation}
denotes the extremal function associated with the smooth form \( t\alpha + (1-t)\omega_0 \).  
See \cite{BEGZ, fgs, GPTW, GPSS} for related results.  
To the best of the author’s knowledge, such an optimal lower bound~\eqref{eq--optimal lower bound for C0} remains unknown in the volume-collapsing finite-time singularity case of the Kähler–Ricci flow.
\end{remark}

\noindent \textit{Proof of Theorem \ref{thm-generalized schwarz}.}
Since $X$ is smooth, every irreducible component of the exceptional set of $\mu$ has codimension 1. Let $E$ denote the reduced effective exceptional divisor of $\mu:Z\rightarrow X$.

We first give the proof under the extra assumption that $\supp(E)\subset \supp(f^*D_Y)$. To proceed as in the standard proof of the Schwarz lemma, we need the following fact to handle the additional difficulties arising from the nontrivial bimeromorphic morphism \( \mu \).

\begin{itemize}
    \item[] Since $\mu$ is a finite composition of blow-ups along holomorphic submanifolds, by a local computation, we know that for a K\"ahler form $\omega_Z$ on $Z$ there exists a constant $C_0>0$ such that on $Z\setminus \Exc(\mu)$, we have 
   $ |s_E|^{C_0}\tr_{\mu^*(\omega_0)}(\omega_Z)\leq C_0 $.

\end{itemize}Then under the extra assumption $\supp E\subset \supp(f^*D_Y)$, we can derive that there exists 
a constant $C_1$ such that:
 for each fixed $t\in [0,1)$, 
 \begin{equation}\label{eq-deneracy control}
|f^*s_{D_Y}|^{C_1}_h\tr_{\mu^*(\omega_t)}(f^*(\theta_1))\leq C_t |s_E|^{C_0}\tr_{\mu^*(\omega_0)}(\omega_Z)\leq C_t
 \end{equation} is a globally bounded function on $X\setminus \Exc(\mu)=X\setminus\supp(E)$ .

By our assumption, we know that on $Z$, the class  the class $\mu^*([\omega_0+K_X])$ admits a smooth representative
\begin{equation}
    \alpha_Z:=f^*(\theta_1)+\frac{\ii}{2\pi}\Theta_{h}+\theta_2 \text{ with } \frac{\ii}{2\pi}\Theta_{h}\in c_1(D),
\end{equation}which on $Z\setminus \supp(D)$ is given by 
\begin{equation}
    f^*(\theta_1)-\frac{\ii}{2\pi}\p\pp \log |s_D|^2_{h}+\theta_2.
\end{equation} We can write 
\begin{equation}
    \mu^*(\omega_t)=(1-t)\mu^*(\omega_0)+t\alpha_Z+\frac{1}{2\pi}\ii\p\pp \phi_t. 
\end{equation}By Lemma \ref{lem--C0 estimate}, we know that there exists a constant $C$, independent of $t$ such that 
\begin{equation}\label{eq--estimate on the C0 and time derivative norm}
    t\log |s_D|^2_h-C\leq \phi_t\leq C \quad \text{ and }\quad \partial_t\phi\geq \mu^*\log(\frac{\omega_t^n}{\omega_0^n})-C.
\end{equation}
Then using our assumptions on the divisors $D$, we can prove a key improved lower bound of $\phi_t$.
Recall the assumption 
\begin{equation}
\mu^*(K+\omega_0)=f^*(\alpha_1)+\alpha_2+D, \quad    D+F=f^*(D_Y).
\end{equation}Note that since $\alpha_1$ is a K\"ahler class on $Y$, we know that for $0<\epsilon\ll 1$, $\alpha_1+\epsilon D_Y$ is still a K\"ahler class on $Y$. We fix a small $\epsilon_0\in \mathbb R_{>0}$ Then using $F$ is $\mu$-exceptional and applying the negativity lemma to $\mu:Z\rightarrow X$, using
\begin{equation}
    \mu^*(K+\omega_0)=f^*(\alpha_1+\epsilon_0 D_Y)+f^*((1-\epsilon_0)D_Y)-F,
\end{equation}we get that 
\begin{equation}\label{eq--D dominate D_Y}
D_{\epsilon_0}:=f^*((1-\epsilon_0)D_Y)-F\geq 0.
\end{equation}Then in particular, we know that there exists an $\alpha_Z$-psh function $\frac{\eta}{2\pi}$ such that 
\begin{equation}\label{eq-lower of extremal function}
    \eta\geq \log |s_{D_{\epsilon_0}}|^2-C.
\end{equation}

Note that by definition, $ D=\epsilon_0 f^*(D_Y)+f^*((1-\epsilon_0)D_Y)-F=\epsilon_0 (D+F)+D_{\epsilon_0}$, which by \eqref{eq--D dominate D_Y}, implies
\begin{equation}\label{eq--divisor relation}
   (1-\epsilon_0)D\geq D_{\epsilon_0} \text{ and } D\geq \epsilon_0 f^*(D_Y).
   \end{equation}Then combining \eqref{eq-lower of extremal function} and \eqref{eq--divisor relation}, we know that  
\begin{equation}
\eta\geq (1-\epsilon_0)\log |s_{D}|^2-C
\end{equation}Therefore by Lemma \ref{lem--C0 estimate}, we know that there exists a constant $C$ independent of $t\in [0,1)$ such that 
\begin{equation}\label{eq--improve lower bound of potential}
 \phi_t\geq    (1-\epsilon_0)t\log |s_D|_h^2-C 
\end{equation}


In the following, we perform the computation on $Z\setminus \supp (D+E)$, where $\mu^*(\omega_t)$ is still a K\"ahler form for $t\in [0,1)$. Let 
\begin{equation}
    u=\tr_{\mu^*(\omega)}(f^*(\theta_1))
\end{equation}Then a standard computation \cite[Theorem 3.2.6]{song-weinkove} shows that 
\begin{equation}\label{eq-parbolic schawarz}
    (\partial_t-\Delta)\log u\leq C_1u.
\end{equation}Since $\theta_2$ is smooth and semi-positive, we have 
\begin{equation}\label{eq-contribution from the kahler potential}
    \Delta(\phi_t-t\log |s_D|^2_{h_D})\leq n-tu-(1-t)\mu^*(\tr_{\omega_t}\omega_0)
\end{equation}
Let 
\begin{equation}
    \begin{aligned}
        Q=&\log u-A(\phi_t-t\log |s_D|^2_{h_D})-An (1-t)(\log (1-t)-1),
    \end{aligned}
\end{equation}
where $A=\epsilon_0^{-2}(4C_1+10)$ is a fixed large constant. 
Then by \eqref{eq--estimate on the C0 and time derivative norm}, \eqref{eq-contribution from the kahler potential}, and the choice of $A$, we obtain that for $t\geq \frac{1}{2}$, there exists a constant $C$ such that 
\begin{equation}
\begin{aligned}
      (\partial_t-\Delta)Q\leq & (\partial_t-\Delta)\log u+A\Delta(\phi_t-t\log |s_D|_{h_D}^2)-A\p_t\phi_t+An\log (1-t)+C\\
      &\leq -u+C-A\partial_t\phi_t-A(1-t)\tr_{\mu^*(\omega_t)}\mu^*(\omega_0)+An \log (1-t)\\
      &\leq -u+C-A\mu^*\left(\left(\frac{((1-t)\omega_0)^n}{\omega_t^n}\right)^{\frac{1}{n}}-\log( \frac{\omega_t^n}{((1-t)\omega_0)^n})\right).
\end{aligned}
\end{equation} By the divisor relation $D\geq \epsilon_0 f^*(D_Y)$ in \eqref{eq--divisor relation} and the lower bound estimate \eqref{eq--improve lower bound of potential} on $\phi_t$, we know that 
\begin{equation}\label{eq--bounded above property}
    Q\leq \log (|f^*s_{D_Y}|^{C_1}u)-A(\phi_t-(1-\epsilon_0)\log |s_D|_h^2)+C \text{ is bounded above }
\end{equation}and  
its maximum can only be attained on \( Z \setminus \operatorname{Supp}(D + E) \).
 Note that the function $x\rightarrow x^{\frac{1}{n}}-\log x$ is bounded below for $x>0$. For any $t'\in (1/2,1)$, Applying the maximum principle to \( Q \) on \( Z \times [\tfrac{1}{2}, t'] \), we obtain that there exists a constant \( C \), independent of \( t' \), such that for any \( x \in Z \setminus \operatorname{Supp}(D + E) \),
\begin{equation}
    u(x) \leq C \, |s_D|_{h_D}^{-C}\leq C|f^*s_{D_Y}|_{h_{D_Y}}^{-C}.
\end{equation}
Letting \( t' \to 1^- \), we then conclude the proof of the desired estimate in Theorem~\ref{thm-generalized schwarz} under the extra assumption $\supp E\subset \supp(f^*(D_Y))$.

We now prove the general case.  
Consider the Zariski open set 
\( U = Z \setminus f^{-1}(\operatorname{Supp}(D_Y)) \) 
and its image under \( \mu \), denoted by \( \mu(U) \subset X \).  
For any \( p \in \mu(U) \) and any curve \( C \subset \mu^{-1}(p) \),  
the restriction of the class \( \mu^*(K_X + \omega_0) \) to \( C \) is trivial.  
Moreover, since \( \mu^{-1}(p) \cap U \neq \emptyset \), the image \( f(\mu^{-1}(p)) \) 
is not contained in \( \operatorname{Supp}(D_Y) \).  
Therefore, by the class identity \eqref{eq-class decomposition} 
and the assumption that \( \alpha_1 \) is Kähler, we conclude that  
\( f(C) \) must be a point contained in \( Y \setminus \operatorname{Supp}(D_Y) \).  
This implies that 
\begin{equation}
    \mu^{-1}(\mu(U)) = U.
\end{equation}
In particular, by the proper mapping theorem, 
\( \mu(U) = X \setminus \mu(f^{-1}(\operatorname{Supp}(D_Y))) \) 
is also Zariski open.  
Furthermore, by \cite[Lemma~1.15]{debarre}, the morphism \( f \) 
restricted to \( Z \setminus f^{-1}(\operatorname{Supp}(D_Y)) \) 
factors through \( \mu \); that is, the meromorphic map  
\( \pi : X \dashrightarrow Y \) is well defined on \( \mu(U) \).  
Hence, on \( U \), we can regard the function 
\[
    \operatorname{tr}_{\mu^*(\omega_t)}(f^*(\theta_1))
    = \mu^*\bigl(\operatorname{tr}_{\omega_t}(\pi^*(\theta_1))\bigr),
\]
which is smooth on \( U \).

By considering the graph of the meromorphic map \( \pi \) and applying the elimination of points of indeterminacy, 
we obtain a modification \( \nu_1 \colon X' \rightarrow X \), which is a finite composition of blow-ups along smooth submanifolds contained in the complement of \( \mu(U) \), 
such that there exists a holomorphic morphism \( \pi' \colon X' \rightarrow Y \).  
More precisely, we can write 
\[
\nu_1 = \psi_k \circ \cdots \circ \psi_1,
\]
where each \( \psi_i \colon X_i \rightarrow X_{i-1} \) is the blow-up along a smooth submanifold \( S_i \) satisfying 
\[ S_i \cap (\psi_{i-1} \circ \cdots \circ \psi_1)^{-1}(U) = \emptyset, \] where $X_0=X$ and $X'=X_k$.
Let \( E' \) denote the effective reduced exceptional divisor of \( \nu_1 \).  
Since \( \nu_1 \) is a finite composition of blow-ups along smooth centers, there exists a Kähler form \( \omega_{X'} \) on \( X' \) and a large constant \( C > 0 \) such that 
\begin{equation}\label{eq-first resolution estimate}
    \tr_{(\nu_1)^*(\omega_t)}\bigl((\pi')^*(\theta_1)\bigr)
    \leq \operatorname{tr}_{(\nu_1)^*(\omega_t)}(\omega_{X'})
    \leq C_t \, |s_{E'}|^{-C}.
\end{equation}

Next, we consider the induced meromorphic map from \( Z \) to \( X' \), which is well-defined on \( U = Z \setminus f^{-1}(\operatorname{Supp}(D_Y)) \) by the construction of $X'$
Applying the elimination of indeterminacies once more, we obtain a modification 
\( \nu_2 \colon Z' \rightarrow Z \), which is a finite composition of blow-ups along smooth submanifolds contained in the complement of \( U \), 
and a holomorphic morphism \( \mu' \colon Z' \rightarrow X' \).  
In summary, we have the following commutative diagram:
 \begin{equation}
         \begin{tikzcd}[column sep=large,row sep=large]
X' \arrow[d, "\nu_1"'] \arrow[dr, "\pi'"'] & Z' \arrow[l, "\mu'"'] \arrow[d, "f'"'] \arrow[r, "\nu_2"] & Z\arrow[dl, "f"]\\
X \arrow[r, "\pi"',dashed]& Y  &.
\end{tikzcd}
    \end{equation}Then by \eqref{eq-first resolution estimate}, we know that 
    \begin{equation}\label{eq--on Z'}
      \nu_2^*\left( \tr_{\mu^*\omega_t}(f^*(\theta_1))\right)= \tr_{(\nu_1\circ \mu')^*(\omega_t)}((f')^*(\theta_1))\leq (\mu')^*\left( \tr_{(\nu_1)^*(\omega_t)}\bigl((\pi')^*(\theta_1)\bigr)\right)\leq C_t |(\mu')^*s_{E'}|^{-C}.
    \end{equation}By the construction of $\nu_1$ and $\nu_2$, we know that 
    \begin{equation}\label{eq-divisor relation on Z'}
        \supp((\mu')^*(E'))\subset \supp((\nu_2)^*(f^*D_Y)) 
    \end{equation} Therefore combining \eqref{eq--on Z'} and \eqref{eq-divisor relation on Z'} and by choosing a larger constant $C$, we have
    \begin{equation}
         \nu_2^*\left( \tr_{\mu^*\omega_t}(f^*(\theta_1))\right)\leq C_t|s_{\nu_2^*(f^*{D_Y})})|^{-C},
    \end{equation}
     and therefore on $Z$, we obtain 
    \begin{equation}\label{eq-key global boundedness}
 \tr_{\mu^*\omega_t}(f^*(\theta_1))\leq C_t |s_{f^*D_Y}|^{-C}\leq C_t|s_D|^{-C},        
    \end{equation}where we have used \eqref{eq--divisor relation} for the second inequality.
  Once \eqref{eq-key global boundedness} is established—thereby verifying \eqref{eq-deneracy control} and \eqref{eq--bounded above property} in the general case—we can proceed as in the previous case to obtain the desired estimate.
\qed

\

\noindent\textit{Proof of Theorem \ref{thm--main}.}
We apply Theorem~\ref{thm-weak basepointfree} to obtain projective fibrations between compact Kähler manifolds 
\(
X \xlongleftarrow[]{\mu} Z \xlongrightarrow[]{f} Y,
\) where $\mu$ is a composition of blow-ups along smooth holomorphic submanifolds.
By this result, we know that if \( Y \) is a point, then \( K_X + \omega = 0 \) in \( H^2(X, \mathbb{R}) \), and hence \( (X, \omega_0) \) is Fano.  
Therefore, in the following, we may assume that 
\begin{equation}
    \dim Y > 0.
\end{equation}

Then we can apply Theorem \ref{thm-generalized schwarz} to get that there exist a Kähler form \( \theta \) on \( Y \) and an effective \( \mathbb{Q} \)-divisor \( D_Y \) such that 
\begin{equation}\label{eq-degenerate lower bound}
    \operatorname{tr}_{\mu^*(\omega_t)}(f^*(\theta)) \leq C\, |f^*s_{D_Y}|^{-C}.
\end{equation}
We can then choose an open set \( U \subset Y \), biholomorphic to the standard unit ball in \( \mathbb{C}^{\dim Y} \), such that 
\[
U\cap \supp(D_Y) = \emptyset.
\]
From~\eqref{eq-degenerate lower bound}, it follows that on \( f^{-1}(U) \) there exists a constant \( \delta > 0 \) such that for all \( t \in [0,1) \),
\begin{equation}
    \mu^*(\omega_t) \geq \delta\, f^*(\theta).
\end{equation}
It then follows (see, for example, \cite[Theorem~4.1]{song2014finite}) that this inequality yields a uniform lower bound for the diameter of \( \omega_t \).
\qed

\begin{remark}
Using Theorem~\ref{thm-weak basepointfree} and adapting the arguments from the proofs of Theorems~\ref{thm--main} and \ref{thm-generalized schwarz}, we also obtain a uniform positive lower bound for the diameter along the continuity path considered in \cite{la-tian}.
For a compact Kähler manifold $(X,\omega_0)$ with $\nef_X(\omega_0)=1$, it is shown in \cite{la-tian} that for every $t\in[0,1)$ there exists a solution $\omega_t$ to
\begin{equation}\label{eq:continuity-path}
    \omega_t = \omega_0 - t\ric(\omega_t).
\end{equation}

\begin{proposition}
Suppose $(X,\omega_0)$ is not Fano and $\nef_X(\omega_0)=1$. Let $\omega_t$ be the solution of \eqref{eq:continuity-path}. Then
\[
    \liminf_{t\nearrow 1}\operatorname{diam}(X,\omega_t) \;>\; 0.
\]
\end{proposition}
\end{remark}


\bibliographystyle{alpha}
\bibliography{ref.bib}

\end{document}